\documentclass[10pt]{amsart}
\usepackage{url}

\usepackage{graphicx}
\usepackage{datetime}
\usepackage{color}
\usepackage{amsmath,amsfonts,amsthm,amssymb}
\usepackage{epstopdf}
\usepackage[all,cmtip]{xy}
\usepackage{accents}
\newtheorem{theorem}{Theorem}
\newtheorem{proposition}[theorem]{Proposition}
\newtheorem{definition}[theorem]{Definition}

\newtheorem{remark}[theorem]{Remark}

\newtheorem*{thmr}{Theorem \ref{rigidity}}
\newtheorem*{thmr2}{Theorem \ref{rigidity2}}
\newtheorem*{thmr3}{Theorem \ref{rigidity3}}
\numberwithin{equation}{section} \numberwithin{theorem}{section}

\newcommand{\R}{\mathbb{R}}
\newcommand{\N}{\mathbb{N}}
\newcommand{\circo}{\accentset{\circ}}
\newcommand{\e}{\varepsilon}

\newcommand{\mbb}{\mathbb}
\newcommand{\C}{\mbb{C}}

\newcommand{\mc}{\mathcal}

\def\XXint#1#2#3{{\setbox0=\hbox{$#1{#2#3}{\int}$}
     \vcenter{\hbox{$#2#3$}}\kern-.5\wd0}}

\newcommand{\ov}{\overline}

\DeclareMathOperator{\dist}{dist}

\DeclareMathOperator{\diam}{diam}

\begin{document}
\title[Quantitative rigidity results for conformal immersions]{Quantitative rigidity results for conformal immersions}
\author{Tobias Lamm}
\address[T.~Lamm]{Institute for Analysis, Karlsruhe Institute of Technology (KIT), Kaiserstr. 89-93, D-76133 Karlsruhe, Germany}
\email{tobias.lamm@kit.edu}
\author{Huy The Nguyen}
\address[H.T.~Nguyen]{School of Mathematics and Physics, The University of Queensland, Brisbane QLD 4072, Australia}
\email{huy.nguyen@maths.uq.edu.au}
\thanks{The second author was supported by The Leverhulme Trust. The authors thank Reiner Sch\"atzle for pointing out a gap in a previous version of the paper.}
\date{\currenttime, \today}


\begin{abstract}
In this paper we prove several quantitative rigidity results for conformal immersions of surfaces in $\R^n$ with bounded total curvature. We show that (branched) conformal immersions which are close in energy to either a round sphere, a conformal Clifford torus,
an inverted catenoid, an inverted Enneper's minimal surface or an inverted Chen's minimal graph must be close to these surfaces in the $W^{2,2}$-norm. Moreover, we apply these results to prove a corresponding rigidity result for complete, connected and non-compact surfaces.     
\end{abstract}
\maketitle
\section{Introduction}

A classical result in differential geometry is a version of Codazzi's theorem which states that every closed, connected and immersed surface $\Sigma \subset \R^n$ with vanishing traceless second fundamental form $\circo A \equiv 0$ is isometric to a round sphere. Such a result is called a rigidity theorem, a curvature condition gives us a classification of the possible geometries. In this paper, we will be interested in quantitative rigidity results, which show that if we relax the conditions of a rigidity theorem, we show that instead of being isometric to fixed object, we are close in a suitable norm. For surfaces in $ \R^3$, Codazzi's theorem was improved by DeLellis-M\"uller where they showed the following quantitative rigidity result. 
\begin{theorem}[{\cite[Theorem 1.1]{DeLellis-Muller:2005}}]\label{thm1}
Let $\Sigma \subset \R^3$ denote a smooth compact connected surface without boundary and with area, $\mu(\Sigma) = 4\pi$. If $||\circo A||_{L^2(\Sigma)}^2 \le 8\pi$, then there exists a conformal parameterisation 
$\psi: \mbb S^2 \to \Sigma$ and a vector $c_\Sigma\in \R^3$ such that 
\begin{align}
 ||\psi -(c_\Sigma+id)||_{W^{2,2}(\mbb S^2,\R^3)} \le C ||\circo A||_{L^2(\Sigma)}, \label{delm}
\end{align}
where $C$ is a universal constant.
\end{theorem}
The assumption $||\circo A||_{L^2(\Sigma)}^2 \le 8\pi$ implies that $\Sigma$ is topologically a sphere. A result of Jost \cite{jost91} (see also Morrey \cite{morrey66}) then yields the existence of a conformal parameterisation $\tilde{\psi}:\mbb S^2 \to \Sigma$.
The major contribution of DeLellis-M\"uller was to show that there exists a conformal parameterisation satisfying \eqref{delm}. An important ingredient in the proof were the analytical results of M\"uller-{\v{S}}ver{\'a}k \cite{Muller1995} for surfaces with finite total curvature. We remark that DeLellis-M\"uller used the Codazzi-Mainardi equations and this is one instance were the smoothness of the surface was necessary.
Finally, we mention that DeLellis-M\"uller \cite{DeLellis-Muller:2006} extended their analysis to show that the conformal factor of the pull-back metric by $\Psi$ of $\Sigma$ is $L^\infty$-close to $1$. Moreover, higher dimensional subcritical variants of the Theorem \ref{thm1} have recently been obtained in \cite{Perez}.

Applications of the Theorem \ref{thm1} include the construction of foliations of asymptotically flat and hyperbolic manifolds by surfaces of constant or prescribed mean curvature \cite{Metzger:2007ce}, \cite{NevesTian} and by surfaces of Willmore type \cite{lamm-metzger-schulze:2009}.
See also \cite{LammMetzger2010}, \cite{LammMetzger2012} where the above theorem has been used in order to construct and locate minimisers of the Willmore functional with small area in closed Riemannian 3-manifolds. 

Recently, interesting intrinsic stability results have been obtained in \cite{DeLellis-Topping:2012}, \cite{GeWang2010} and \cite{GeWang2011}. Other examples of very important stability results include the harmonic approximation lemma \cite{Simon1996} and the geometric rigidity result of \cite{friesecke02}.

We note that for immersions $f:\Sigma \to \R^n$ the Willmore functional is defined by
\[
 \mc W (f)= \frac14 \int_\Sigma |H|^2 d\mu_g.
\]
Using the Gauss equations 
\begin{align}
\frac14 |H|^2-\frac12 |\circo A|^2=\frac12 (|H|^2-|A|^2)=K \label{gauss} 
\end{align}
and the Gauss-Bonnet theorem, we can rewrite the Willmore functional as follows
\begin{align}
 \mc W(f)= \frac14 \int_\Sigma |A|^2 d\mu_g+\pi \chi(\Sigma)=\frac12 \int_\Sigma |\circo A|^2 d\mu_g +2\pi \chi(\Sigma), \label{Willeq}
\end{align}
where $\chi(\Sigma)$ denotes the Euler characteristic of $\Sigma$.

Hence, in the case $\chi(\Sigma)=2$, the assumption $||\circo A||_{L^2(\Sigma)}^2 \le 8\pi$ is equivalent to the fact that $\mc W(f) \le 8\pi$ or $\int_\Sigma |A|^2 d\mu_g \le 24\pi$.
\newline
In this paper, we will extend Theorem \ref{thm1} to non-smooth surfaces in higher codimensions. In order to do this we work in the class $W^{2,2}_{conf}$ of conformal immersions of a surface. Roughly speaking these are $W^{2,2}$ immersions $f:\Sigma \to \R^n$ for which the pullback metric in local conformal coordinates is conformal to the Euclidean metric with a conformal factor which belongs to $L^\infty$.  
Our first main result is the following
\begin{theorem}\label{rigidity}
There exists $\delta_0>0$ such that for every $0<\delta<\delta_0$ and every immersion $f\in W^{2,2}_{\text{conf}}(\mbb S^2, \R^n)$ with conformal factor $u$, satisfying $\int_{\mbb S^2} |A|^2 d\mu \le 8\pi+\delta$, there exists a constant $C(\delta)$, with $C(\delta)\to 0$ as $\delta \to 0$, and a standard immersion $f_{\text{round}}\in W^{2,2}_{\text{conf}}(\mbb S^2,\R^n)$ of a round sphere such that 
\begin{align}
||f -f_{round}||_{W_u^{2,2}(\mbb S^2,\R^n)}  \le C(\delta). \label{closeness}
\end{align} 
\end{theorem}  
For the definition of the norm $W^{2,2}_u(\mbb S^2,\R^n)$ we refer the reader to \eqref{norm}. 
We also refer the reader to Remark \ref{compare} where we compare the assumptions and results of Theorem \ref{thm1} and Theorem \ref{rigidity}. Note that $\int_{\Sigma} |A|^2 d\mu \ge 8\pi$ for all immersions of a closed surface and equality holds if and only if $\Sigma$ is a round sphere.

Contrary to the proof of Theorem \ref{thm1}, we argue by contradiction in order to show \eqref{closeness}. Therefore, we are naturally led to studying sequences of conformal immersions of the sphere with uniformly bounded total curvature.
In a recent paper Kuwert-Li \cite{Kuwert2010} showed that such a sequence converges weakly (modulo the composition with suitably chosen M\"obius transformations) in $W^{2,2}_{loc}$ away from an at most finite set of points to a branched conformal immersion (see section $2$ for a definition of branched conformal immersions). In order to prove Theorem \ref{rigidity} 
we show that in our situation this convergence can be improved to everywhere strong convergence in $W^{2,2}$. The limit is shown to be a standard smooth immersion of a round sphere which gives the desired contradiction. 

Note that after this paper has been finished, the first author and Sch\"atzle extended the results of DeLellis and M\"uller to arbitrary codimensions, see \cite{LS}.

To summarise, the main ingredients in the proof are convergence results for sequences of (branched) conformal immersions, an understanding of the possible singular set of the limit (in the above situation the singular set is empty) together with a classification result for the limiting (branched) conformal immersion.

This method is flexible and has applications to several other problems. In particular, assuming the validity of the generalised Willmore conjecture, we show a rigidity result for minimisers of the Willmore energy of arbitrary genus in $\R^n$ (see Theorem \ref{genus}). In particular, we show a rigidity result for the Clifford torus in $\R^3$ within its own conformal class. 

We extend the method to handle conformal immersions of the sphere in $\R^3$ with either exactly one point of multiplicity two and two preimage points or exactly one branch point of branch order two. 
Using an inversion formula for the Willmore energy derived in \cite{Nguyen2011} (see also Theorem \ref{thm_three}), we can show that $\int |A|^2 d\mu$ of these surfaces is always bigger than or equal to $24\pi$, respectively $32\pi$ and by a classification result of Osserman \cite{Osserman1964} the infimum is attained by an inversion of the catenoid respectively an inversion of Enneper`s minimal surface.
More precisely, we prove the following two theorems, firstly for multiplicity two points we have that
\begin{theorem}\label{rigidity2}
There exists $\delta_{cat}=\delta_0>0$ such that for every $0<\delta<\delta_0$ and every immersion $f\in W^{2,2}_{\text{conf}}(\mbb S^2, \R^3)$ with conformal factor $u$ which has exactly one point of multiplicity two $ x \in f ( \mbb{S}^2)$ with $ f^{-1} (x) = \{ p_{1}, p_{2} \}$ and which satisfies $24 \pi \leq  \int_{\mbb S^2} |A|^2 d\mu \le 24\pi+\delta$, there exists a M\"obius transformation $\sigma:\R^n \to \R^n$, a reparameterisation $\phi:\mbb S^2 \to \mbb S^2$, a
constant $C(\delta)$, with $C(\delta)\to 0$ as $\delta \to 0$, and a standard immersion $f_{\text{cat}}\in W^{2,2}_{\text{conf}}(\mbb S^2,\R^3)$ of an inverted catenoid with
\begin{align}
||\sigma \circ f \circ \phi-f_{cat}||_{W_u^{2,2}(\mbb S^2,\R^3)} \le C(\delta)\sqrt{\mu_g(\mbb S^2)}. \label{close}
\end{align} 
\end{theorem}  

And for a branch point of multiplicity three,
\begin{theorem}\label{rigidity3}
There exists $\delta_{Enn}=\delta_0>0$ such that for every $0<\delta<\delta_0$ and every immersion $f\in W^{2,2}_{\text{conf}, br }(\mbb S^2, \R^3)$ with conformal factor $u$ where $ f $ has exactly one branch point $p\in \mbb S^2$ of branch order $m(p) = 2 $, and which satisfies $32 \pi \leq  \int_{\mbb S^2} |A|^2 d\mu \le 32\pi+\delta$ then there exists a M\"obius transformation $\sigma:\R^n \to \R^n$, a reparameterisation $\phi:\mbb S^2 \to \mbb S^2$, a
constant $C(\delta)$, with $C(\delta)\to 0$ as $\delta \to 0$, and a standard immersion $f_{\text{Enn}}\in W^{2,2}_{conf,br}(\mbb S^2,\R^3)$ of an inverted Enneper's minimal surface with
\begin{align}
||\sigma \circ f \circ \phi-f_{Enn}||_{W_u^{2,2}(\mbb S^2,\R^3)} \le C(\delta)\sqrt{\mu_{g}(\mbb S^2)}. \label{Enn_close}
\end{align} 
\end{theorem}
In higher codimensions we show a similar quantitative rigidity result for conformal immersions of the sphere with exactly one branch point of branch order one and Willmore energy close to the Willmore energy of an inversion of Chen's minimal graph (see Theorem \ref{rigidity4} for details).

Finally, in section 5, we combine the rigidity theorems in order to prove a quantitative rigidity result for complete surfaces with finite total curvature. M\"uller-{\v{S}}ver{\'a}k \cite{Muller1995} showed that all complete, connected and non-compact surfaces in $\R^n$ with $ \int _{ \Sigma} |A| ^ 2 d \mu < 8\pi$ for $n=3$ and $ \int _{ \Sigma} |A| ^ 2 d \mu \le 4\pi$ for $n\ge 4$, are embedded. This was extended by the second author in \cite{Nguyen2011} were it was shown that if one allows the inequality $ \int _ {\Sigma} |A|^2 d \mu \leq 8 \pi$ for $n=3$, then the surface is either embedded and conformal to the plane or isometric to a catenoid or Enneper's minimal surface. 
\begin{theorem}[\cite{Muller1995}, Corollary 4.3.2 and \cite{Nguyen2011}, Theorem 14]\label{1}
Let $ f:\Sigma\hookrightarrow \R^{n}$ be a complete, connected, non-compact surface immersed into $ \R^{n}$. Assume that $n=3$ and
\begin{align*}
\int_{\Sigma} |A|^{2}d\mu \leq 8 \pi.
\end{align*}
Then either $ \Sigma$ is embedded and conformal to the plane or isometric to a catenoid or not embedded and $ \Sigma$ is Enneper's minimal surface.

Assuming $n\ge 4$ and 
\begin{align*}
\int_{\Sigma} |A|^2 d\mu  \leq 4 \pi
\end{align*}
we have that $\Sigma$ is embedded and is either conformal to a plane or isometric to Chen's minimal graph.
\end{theorem}
We show that if we assume for some small enough $\delta>0$ the energy bounds $\int_{\Sigma} |A|^{2}d\mu \leq 8 \pi+\delta$ for $n=3$, respectively $\int_{\Sigma} |A|^{2}d\mu \leq 4 \pi+\delta$ for $n\ge 4$, then $\Sigma$ is either embedded and conformal to a plane, or, after an inversion and modulo a M\"obius transformation and a reparameterisation, the surface is $W^{2,2}$ close to an inverted catenoid or an inverted Enneper's minimal surface for $n=3$, respectively an inverted Chen's minimal graph for $n\ge 4$.
\newline
In the following we give a brief outline of the paper.

In section $2$ we recall the definition of conformal and branched conformal immersions. Moreover we restate an inversion formula for surfaces with branch points and ends and we mention various classification results for minimal surfaces which we need later on.

In section $3$ we collect a series of convergence results for branched conformal immersions. These results are partial extensions of earlier results of H{\'e}lein \cite{H'elein2002} and Kuwert-Li \cite{Kuwert2010} however we obtain new strong convergence results under the assumption of no energy loss.

In section $4$ we prove Theorems \ref{rigidity}--\ref{rigidity3} and the above mentioned rigidity result for higher genus minimisers of $\mc W$ and inversions of Chen's minimal graph.

In section $5$ we show the rigidity result for complete, connected and non-compact surfaces immersed in $\R^n$ satisfying certain bounds on the total curvature.

\section{Preliminaries}
\begin{definition}\label{confsp}
Let $ \Sigma$ be a Riemann surface. A map $ f \in W^{2,2}_{loc}( \Sigma, \R^n)$ is called a conformal immersion if in any local conformal coordinates $ (U, z )$, the metric $ g_{i j} = \langle \partial _i f , \partial_j f \rangle $ is given by
\begin{align*}
g_{ij} = e^{2u} \delta _{ij}, \quad u \in L^\infty_{loc} (U).
\end{align*} 
The set of all $ W^{2,2}$-conformal immersions of $\Sigma$ is denoted $W^{2,2}_{conf} ( \Sigma, \R^n)$. 
\end{definition}
For $W^{2,2}$ conformal immersions we have the weak Liouville equation, 
\begin{align*}
\int_{U} \langle Du, D\varphi \rangle = \int_{U} K_g e ^ {2u} \varphi \quad \forall \varphi \in C^\infty_0(U).
\end{align*}
This is shown to hold in \cite{Kuwert2010}. 

\begin{definition}[Branched conformal immersion]  \label{defn_branched}
 A map $ f \in W^{2,2} (\Sigma, \R^{n})$ is called a branched conformal immersion (with locally square integrable second fundamental form) if $ f \in W ^ {2,2}_ {conf} ( \Sigma \backslash \mc{S}, \R^n)$ for some discrete set $\mc S \subset \Sigma$ and if for each $ p \in \mc {S}$ there exists a neighbourhood $ \Omega_p $ such that in local conformal coordinates
 \begin{align*}
\int_{\Omega_p\backslash \{p\} } |A|^2 d\mu_g < \infty.   
\end{align*}
Additionally, we either require that $ \mu_{g} (\Omega_p \backslash \{p\} )< \infty $ or that $ p $ is a complete end. 

We denote the space of branched conformal immersions by $W^{2,2}_{conf,br}(\Sigma,\R^n)$.
\end{definition}  
We remark that suitable inversions send complete ends to finite area branch points and therefore we only consider finite area branch points in the rest of this paper.

Now we let $f:\Sigma \to \R^n$ be a branched conformal immersion and we let $p\in \Sigma$. We choose a punctured neighbourhood of $p$ which is conformally equivalent to $D\backslash \{0\}$ where $D=\{ z\in \R^2 |\,\ |z|<1\}$. 
Thus $f$ is a $W^{2,2}$-conformal immersion from $D\backslash \{0\}$ to $\R^n$ with $\int_{D\backslash \{0\} } |A|^2 d\mu+\mu_g(D\backslash \{0\})<\infty$ and hence a result of Kuwert-Li \cite{Kuwert2010} (see also Lemma A.4 in \cite{Rivi`ere2010}) shows the existence of a number $m\in \N_0$ such that $\theta^2 (\mu,0) =m+1$, where 
\[
 \mu:=f(\mu_g)=\Big(x\mapsto \mathcal{H}^0(f^{-1}(x))\Big) \mathcal{H}^2 \lfloor f(\Sigma)
\]
is the weight measure of the integral $2$-varifold associated with $f$ (see e.g. \cite{Kuwert2004}).

By another result of Kuwert-Li \cite{Kuwert2010} we have for all $z\in D\backslash \{0\}$
\begin{align}
 u(z)=& m\log |z| +\omega(z) \ \ \ \text{and} \label{expansionu} \\
-\Delta u =& -2\pi m \delta_0 +K_ge^{2u} \label{laplaceu},
\end{align}
where $\omega \in C^0\cap W^{1,2} \cap W^{2,1} (D)$ and $\delta_0$ denotes the Dirac delta distribution at the origin. 

Next we recall a generalised Gauss-Bonnet and inversion formula for general surfaces with branch points and ends. Special cases of this result have been obtained previously in \cite{CL:2011}, \cite{Kuwert2010}, \cite{li91} and \cite{Muller1995}.
\begin{theorem}[{\cite[Theorem 4.1, 4.2 and Corollary 4.3, 4.4]{Nguyen2011}}]\label{thm_three}
Let $ f:\Sigma\to  \R^n$ be a branched conformal immersion of a Riemann surface $\Sigma$. We denote by $ E = \{ a_{1}, \dots , a_{b}\}$, $b\in \N_0$, the complete ends with multiplicity $k(a_i)+1$, $1\le i\le b$, and for each 
$p\in \Sigma \backslash E$ we denote by 
$m(p)\in \N_0$ the number $m$ from \eqref{expansionu}. For $x_0\in \R^n$ we let $\tilde{f}=I_ {x_0} \circ f:=x_0+\frac{f-x_0}{|f-x_0|^2}:\Sigma\backslash f^{-1}(x_0) \to \R^n$ and 
we denote by $\tilde{\Sigma}=\tilde{f}(\Sigma \backslash f^{-1}(x_0))$ the image surface of $\tilde{f}$. Then we have the formulas
\begin{align}
 \int_\Sigma K d\mu =& 2\pi \Big( \chi(\Sigma)-\sum_{i=1}^b (k(a_i)+1)+\sum_{p\in \Sigma\backslash E} m(p)\Big),\label{gauss-bonnet}\\
\int_{\widetilde{\Sigma}} \widetilde{K} d\widetilde{\mu}=&\int_\Sigma K d\mu +4\pi \Big(  \sum_{i=1}^b (k(a_i)+1)-\sum_{p\in f^{-1}(x_0)} (m(p)+1) \Big), \label{inversion1}\\
 \mc W ( \tilde f ) =& \mc W (f) + 4\pi \Big(\sum_{i=1} ^b ( k(a_i)+1) -  \sum_{p \in f ^ { -1} ( x_0)} (m(p)+1) \Big)\ \ \ \text{and} \label{inversion2}\\
\int_{\widetilde{\Sigma} }|\widetilde A|^{2}d \widetilde \mu   =&  \int_{\Sigma} |A|^2 d\mu + 8\pi \Big(\sum_{i=1} ^b ( k(a_i)+1) - \sum_{p \in f ^ { -1} (x_0)} (m(p)+1) \Big).\label{inversion3} 
\end{align} 
\end{theorem}
Moreover, we will make use of the following theorem due to Osserman, where complete minimal surfaces with finite total curvature equal to $8 \pi$ are classified. 
\begin{theorem}[{\cite[Theorem 3.4]{Osserman1964}}]\label{thm_osserman}\label{thm_hoss}
Let $ \Sigma\subset \R^3$ be a complete minimal surface with finite total curvature $ \int_{\Sigma} |A|^{2} d\mu= 8\pi$. Then $\Sigma$  is either isometric to Enneper's minimal surface or a catenoid.  
\end{theorem}
In higher codimension there are complete minimal surfaces with $ \int_\Sigma |A|^ 2 d \mu = 4 \pi$. By a result due to Chen \cite{Chen1979}, these surfaces have been classified. In a paper by Hoffman-Osserman \cite{Hoffman1980}, a representation in terms of an orthogonal complex structure was shown. 
\begin{theorem} [\cite{Chen1979},\cite{Hoffman1980} Chen's Minimal Graph]\label{thm_chen}
Let $ \Sigma$ be a minimal surface in $ \R^n$, $n\ge 4$. Suppose that $\Sigma$ is complete and that we have
\begin{align*}
\int_{ \Sigma} |A| ^ 2 d \mu = 4 \pi.
\end{align*} 
Then $ \Sigma$ lies in a four dimensional subspace. Furthermore if we let $ (w,z)$ be co-ordinates in $ \mbb{C} ^ 2 $ then $\Sigma$ is the graph of the function 
\begin{align*}
w = c z ^2
\end{align*}
for some $c\in \mbb{C}$.

Surfaces corresponding to different values of $c$ are not isometric and the surface, being the graph of a function, is embedded.
\end{theorem}

\section{Convergence results}

In this section we derive various weak and strong convergence results for sequences of branched conformal immersions.

\begin{definition}\label{confsp2}
Let $m\in \N_0$. A map $f \in W^{2,2}(D,\R^n)$ is called a $m$-branched conformal immersion, if on $ D \backslash\{0\} $ $ f$ is a conformal immersion and if there exist conformal coordinates about $0$ such that the induced metric is given by
\[
 g_{ij}=e^{2u} \delta_{ij},\ \ \ \text{where} \ \ \ u-m\log |\cdot| \in L^\infty\cap W^{1,2}( D).
\]
We denote the set of all $W^{2,2}$ $m$-branched conformal immersions by $W^{2,2}_{\text{conf,m}}(D,\R^n)$.   
\end{definition}
Note that $0$-branched conformal immersions are in $W^{2,2}_{conf}(D,\R^n)$ and that for every branched conformal immersion $f:\Sigma\to \R^n$ and every point $p\in \Sigma$, there exists a number $m\in \N_0$ such that in conformal coordinates in a small neighbourhood around $p$ the immersion is a $m$-branched conformal immersion by the discussions in section $2$.

In the case $m=0$ the following theorem is due to H{\'e}lein \cite{H'elein2002} with $\gamma_n = \frac{ 8 \pi}{3}$ and with the optimal constant it is due to Kuwert and Li \cite{Kuwert2010}.
\begin{theorem}\label{conv}
Let $f_k \in W^{2,2}_{ conf,m}( D , \R^n) $ be a sequence of m-branched conformal immersions with induced metrics $(g_k)_{ij} = e ^{2u_k} \delta_{ij} $ and assume that 
 \begin{align*}
\int_{ D} | A _{k}| ^2 d \mu_{g_k} \leq \gamma < \gamma_n =\left \{
 \begin{array}{ll}
 8\pi \quad \text { if $n = 3$},\\
 4 \pi \quad \text{ if $n \ge 4$}. 
\end{array}
 \right. 
\end{align*}
Assume also that we have uniformly bounded area $ \mu_{g_k} ( D)\leq C $ and $ f_k(0)=0$. Then $f_k$ is bounded in $ W^{2,2}_{loc}(D, \R^n)$ and there is a subsequence such that one of the following two cases occurs:
\begin{enumerate}
\item $f_k$ converges weakly to a $m$-branched conformal immersion, that is $ u_k-m\log |\cdot|$ is bounded in $ L^{\infty}_{loc} ( D)$ and $ f_k$ converges weakly in $W^{2,2}_{loc} ( D, \R^n)$ to a $m$-branched conformal immersion $f \in W^{2,2}_{loc} (D ,\R^{n})$, or 
\item $f_k$ converges to a constant map, that is $u_k(z) - m \log |z|\rightarrow -\infty$, $\forall z \neq 0$ and $ f_k \rightarrow 0$ locally uniformly on $D$.
\end{enumerate} 
\end{theorem}
 \begin{proof} 
We will follow the proof of Theorem 4.1 in \cite{Kuwert2010}. By \cite{Muller1995} we have that for every $k\in \N$ there exists a solution $v_k:\C\to \R$ of the equation 
 \begin{align*}
 - \triangle v _k = K _{ g _k } e ^{2u_k} \ \ \ \text{in} \ \ D
\end{align*} 
such that 
\begin{align*}
\| v_k\|_{L^\infty(\C) } + \| D v _k \|_{ L ^ 2 (\C)} \leq C( \gamma) \int _{ D } | A_{ f _k}| ^ 2 d \mu _{g_k}. 
\end{align*}
Using the area bound we get
\begin{align*}
 \int _{ D } e ^ { 2 u _k ^+} = | \{ u_k \leq 0 \}| + \int _{ \{u _k > 0 \}  }  e ^ { 2 u _k } \leq C.
\end{align*}  
This gives us the estimate 
\begin{align*}
 \int _{ D} u_k ^+ \leq C.
\end{align*}
Next note that the function $h_k = u _k - v _k - m \log |\cdot| $ is harmonic on $ D$. For $\dist (z, \partial D) \geq r, r \in (0,1) $, we have that 
\begin{align*}
 h _k (z) &= \frac 1 { \pi r ^ 2 } \int _{ D _ r (z)} u_k - v _k - m \log |\cdot| \\
 &\leq \frac 1 { \pi r ^ 2 } \int _{D} ( u_k^+ -m\log |\cdot|) +\| v_k\|_{ L ^ \infty (D)}  \\
&\leq   C(\gamma, r).
\end{align*}
Hence $ u_k - m \log |\cdot| $ is locally uniformly bounded from above, which implies that $e^{2u_k}= |\cdot|^{2m} e ^ {2 h_k +2 v_k }$ is also locally uniformly bounded from above. 
Furthermore, as $ |\nabla f _ k (z)| ^ 2 = 2 e ^ {2u_k(z) } = 2 |z| ^{2m} e ^ {2h_k(z) + 2 v _k(z) } $, we have that $ f _ k $ is bounded in $ W_{loc}^{1,\infty} ( D, \R ^ n)$. The $W^{2,2}$ $m$-branched conformal immersion $f_k$ satisfies 
\begin{align*}
\Delta f _k =e ^{2u_k } H _{ f_k}.  
\end{align*} 
Note that the equation is a priori only satisfied away from the branch point $0$. Over the branch point we can break the equation into components. Therefore any missing contribution must be due to Dirac delta distributions or their derivatives. But $f_k \in W^{2,2}_{loc}(D,\R^n)$ so this can not happen.

Hence we have that on $ \Omega = D_{ 1-r }(0)$,
\begin{align*}
\int_\Omega |\Delta f _k | ^2 &= \int _{ \Omega } e ^ {2 u_k} |H_{ f_k} | ^ 2 d \mu _{ g _k} \leq C( \gamma, r ) \int _{ \Omega} |A_{ f _k } | ^ 2  d \mu _{ g _k}. 
\end{align*}
Therefore, by $L^2$-theory, we see that $ f_k$ is locally uniformly bounded in $W ^{2,2}_{loc}$ and converges weakly in $W^{2,2}_{loc}(D,\R^n)$ to some $ f \in W^{2,2}_{ loc} \cap W_{loc}^{1,\infty} (D, \R^n)$.
\newline
\textbf{Case 1)} $\int _{D} u_k ^ -  \leq C$ 

Then for $ \dist ( z, \partial D) \geq r$ we have that 
\begin{align*}
 h_k(z)& = \frac { 1 }{ \pi r ^ 2 } \int _{ D _ r (z)} u _k - v _k - m \log |\cdot|\\
& \geq - \frac {1}{ \pi r ^ 2 } \int _{ D _ r (z)} u_k^{-} - \| v_k\|_{L^\infty (D)} \\
&\geq - C(\gamma, r ).
\end{align*} 
This shows that $ h_k $ is locally bounded from above and below. Hence $u_k - m \log |\cdot|= h_k + v _k$ is bounded in $ L_{loc}^ \infty \cap W^{ 1,2}_{loc}( D) $. Therefore $ u_k - m\log |\cdot| $ converges pointwise to a function $v \in L ^ \infty _{loc}(D)$ which satisfies 
\begin{align*}
g_{ij} = \langle \partial _i f , \partial_j f \rangle = |z|^{2m}e ^{2v} \delta _{ij}=e^{2u}\delta_{ij},
\end{align*}
where $u:=v+m \log |\cdot|$ and this implies that $f$ is again a $m$-branched conformal immersion.
\newline 
\textbf{Case 2)} $ \int_{D} u_k^- \rightarrow -\infty $ 

In this case we have that 
\begin{align*}
 h_k(0) = \frac 1\pi \int _{ D} u_k - v _k -m \log |\cdot| \rightarrow -\infty.
\end{align*}
Hence, by the Harnack inequality together with the fact that $ C(\gamma, r) - h_k \geq 0 $ in $\Omega $, we have that
\begin{align*}
\sup_{ \Omega'} h _k \leq \frac{1}{ C(\gamma)} \int _{ \Omega' } h _k + C( \gamma,r ) \rightarrow -\infty   
\end{align*}
where $ \Omega' = D_{ 1- 2r }( 0) $. 
This shows that $ u_k - m \log |\cdot| = v _k + h_k  \rightarrow -\infty $ as $ k \rightarrow \infty $ away from the origin. As $ f_k (0)=0$ we see that $ f_k \rightarrow 0 $ locally uniformly.
 \end{proof}

We will also need a global compactness result for sequences of branched conformal immersions. The following result is a minor modification of Proposition $4.1$ in \cite{Kuwert2010}.

\begin{proposition} \label{prop_globalconv}
Let $ \Sigma$ be a closed Riemann surface and let $ f _k \in W ^{2,2} _{ conf, br}( \Sigma, \R ^ n ) $ be a sequence of branched conformal immersions with singular sets $\mc S_k$ satisfying,
\begin{align*}
|\mc S_k|+\int _{ \Sigma} |A_k|^ 2 d \mu _ k \leq \Lambda < \infty. 
\end{align*} 
Then for a subsequence there exists a sequence of M\"obius transforms $ \sigma_k:\R^n \to \R^n $ and an at most finite set $\mc {S}$ such that 
\begin{align*}
\sigma_k \circ f _k \rightharpoonup f \quad  \text{weakly in $ W ^{2,2}_{ loc} ( \Sigma \backslash \mc {S}, \R ^ n )$} ,
\end{align*}
where $ f: \Sigma \rightarrow \R^n$ is a branched conformal immersion with bounded total curvature.  
\end{proposition} 
\begin{proof} 
Since $f_k \in W^{2,2}(\Sigma, \R^n)$ for every $k\in \N$ we can argue as in the proof of Proposition $4.1$ in \cite{Kuwert2010} and conclude that, after performing a suitable dilation, we get the estimate
\[
||u_k||_{W^{1,q}(\Sigma)} \le C
\]
for every $1\le q<2$.

Moreover, away from the possible energy concentration points and the limit of the sets of branch points $\mc S_k$, we also get uniform $W^{2,2}$ bounds for $f_k$.

The rest of the argument is then identical (up to the fact that one also has to use Theorem \ref{conv}) to the proof of Proposition $4.1$ in \cite{Kuwert2010}.   
\end{proof}
\begin{remark} 
It follows from the proof of this Proposition that the M\"obius transformations can either be chosen to be a composition of translations and dilations $\tau_k$ or one has to use an additional inversion $I_{x_0}$ at the boundary of the ball $B_1(x_0)$ for which we have that
\[
B_1(x_0) \cap \tau_k (f_k(\Sigma)) =\emptyset \ \ \ \forall k\in \N.
\] 
\end{remark}
Next we show how the previous two results can be combined in order to conclude that branch points are preserved under the above convergence results. 
\begin{theorem}\label{branchpreserv}
Let $f_k:\mbb S^2 \to \R^n$ be a sequence of branched conformal immersions with exactly one branch point $p_k\in \mbb S^2$ of branch order $m\in \N$ and with uniformly bounded $\int_{\mbb S^2} |A_k|^2 d\mu_{g_k}$. Then there exists a sequence of M\"obius transformations $\sigma_k:\R^n \to \R^n$, a sequence of reparameterisations $\phi_k:\mbb S^2 \to \mbb S^2$ and a branched conformal immersion $f:\mbb S^2 \to \R^n$ with at least one branch point of branch order $m$ and with bounded total curvature, so that 
\begin{align*}
\sigma_k \circ f _k \circ \phi_k \rightharpoonup f \quad  \text{weakly in $ W ^{2,2}_{ loc} ( \mbb S^2 \backslash \mc {S}, \R ^ n )$} ,
\end{align*}
where $p\in \mc{S}$ is the singular set.
\end{theorem}
\begin{proof}
It follows from Proposition \ref{prop_globalconv} that there exists a sequence $\tilde{f}_k:=\sigma_k \circ f_k$ which converges in $W_{loc}^{2,2}(\mbb S^2 \backslash \mc{S},\R^n)$ to a possibly branched conformal immersion $f:\mbb S^2 \to \R^n$. Without loss of generality we can assume that $\tilde{\mu}_k(\mbb S^2)\le C$.
Now we have to consider two cases:
\begin{itemize}
\item[1)] $p:= \lim_{k\to \infty} p_k \notin \mc{S}$ 

In this case it follows from Theorem \ref{conv} that $p$ is a branch point of $f$ with branch order $m$.
\item[2)] $p\in \mc{S}$

In this case we let $0<\tau <\pi$ and we choose local conformal coordinates for the $\tilde{f}_k$ around $p\in \mbb S^2$.  Without loss of generality we assume that these conformal coordinates contain the unit disc $D$ and the point $p$ corresponds to $0\in D$. Moreover we assume that $\tilde{f}_k(0)=0$. 
Next we choose $ r_ k$ so that all discs $ D _ {r_k}(z) $ have a minimal amount of energy, that is 
\begin{align*}
r _ k = \inf \left \{ r\in \left(0, \frac 12 \right]  \bigg| \sup_{z\in D_{\frac12}} \,\ \int _{ D _{ r } (z)} |\tilde{A}_{k}| ^ 2 d \mu _{\tilde{g}_k}= \gamma _n -\tau \right\}.
\end{align*} 
By the assumption $p\in \mc{S}$,  we have that $ r _k \rightarrow 0$ as $ k \rightarrow \infty $. Then let us consider the rescaling 
\begin{align*}
\ov f _k (z)=  \frac { \tilde{f} _k( r _k z )  }{ \lambda _k}
\end{align*}
so that $ | \nabla \ov f _k | ^ 2 = 2 e ^ { \ov u _k} $ and we choose $ \lambda _k$ so that 
\begin{align*}
\int _{ D_ 1 } (\ov u _k(z)-m\log|z|)\,\ dz = 0.
\end{align*} 
As $\tilde{u}_k-m\log|\cdot|$ and $ \ov u _ k-m\log|\cdot|$ differ only by a constant, and $ \int_D |\tilde{A}_{ k} | ^ 2 d \mu _{\tilde{g}_k} $ is scale and translation invariant this implies that 
\begin{align*}
\| \ov u_k -m\log |\cdot| \|_{L ^ \infty ( D_1 ) } +\| \nabla (\ov u _k-m\log |\cdot|)\|_{ L ^ 2 ( D _ 1 )} \leq C .  
\end{align*}
This shows that $\ov f _k $ converges either locally weakly in $ W ^{2,2} ( \mbb C \backslash \{0\}, \R ^ n )$ to a $m$-branched conformal immersion $\ov f:\C\to \R^n$ or $\ov f_k$ collapses to a point. Note that we can not have energy concentration points for the sequence $\ov f_k$ by our blow-up construction.

It is easy to check that 
\begin{align*}
\int _{ D _1  } | |r_kz|^{-2m}\nabla f _ k ( r  _k z)| ^ 2 = \int_{ D _1 }  e ^{ 2 ( u _k (r _k z ) -m\log|r_k z|) + 2 \log r _k }\leq C. 
\end{align*}
As $ \ov u _k(z) = u _k   (r _k z )  + \log r _k - \log \lambda_k $ we see that by Jensen's inequality 
\begin{align*}
\int _{ D _ 1 } 2 ( u _ k (r _ k z ) -m\log|r_k z|+\log r_ k) \leq C,  
\end{align*} 
then we get that $ \sup \lambda_k < \infty$. In particular this shows that the sequence $\ov f_k$ can not collapse to a point. Arguing as in the proof of Theorem $10$ in \cite{Nguyen2011b} we see that $\ov f$ is complete and hence, after another inversion, it can be completed as a branched conformal immersion from $\mbb S^2$ into $\R^n$ with at least one branch point of branch order $m$.
\end{itemize}
\end{proof}

In the next proposition we show that if there is no loss of $\int_D |H|^2 d\mu$ then the weak convergence in Theorem \ref{conv} becomes strong convergence. Note that similar results were also proved in Proposition 5.3 of \cite{Kuwert2013} and in a remark after Proposition 6.1 in \cite{Sch13}. 

Before stating this result we define for every $f\in W^{2,2}_{ conf,m}( D , \R^n)$ (with the obvious modifications for $f\in W^{2,2}_{conf,br}(\mbb S^2,\R^n)$) the norm
\begin{align}
||f||^2_{W^{2,2}_u(D,\R^n)}:=\int_D \left(|f|^2+ |\nabla f|^2+|\nabla^2 f|^2\right)e^{-2u}\,\ dx.\label{norm}
\end{align}
The reason why this norm arises naturally in the problem under consideration is its invariance under the composition of $f$ with any M\"obius transformation of $\R^n$.
\begin{proposition}\label{lem_strong} 
Let $\gamma_n$ be as above and let $f_k \in W^{2,2}_{ conf,m}( D , \R^n) $ be a sequence of $m$-branched conformal immersions with induced metrics $(g_k)_{ij} = e ^{2u_k} \delta_{ij} $ and assume that 
 \begin{align*}
\int_{ D} | A _{k}| ^2 d \mu_{g_k} \leq \gamma <\gamma_n .
\end{align*}
Assume also that  we have uniformly bounded area $ \mu_{g_k} ( D)\leq C $ and $ f_k(0)=0$ and let us assume that $f_k$ does not converge to a constant map. Let us furthermore assume that for a subsequence we have
\begin{equation}\label{conserve2} 
 \int_{K }|H_{k } |^2 d \mu _ {g _ k} \rightarrow \int _{ K} | H | ^ 2 d \mu _g \ \ \ \forall \,\ K \subset \subset D.  
 \end{equation}
Then there exists a $m$-branched conformal immersion $f\in W^{2,2}_{ loc}( D , \R^n) $ such that for every $K \subset \subset D$ we have
\[
||f_k-f||_{W^{2,2}_{u_k}(K,\R^n)} \to 0.
\]
\end{proposition} 
\begin{proof} 
From Theorem \ref{conv} we get the existence of a $m$-branched conformal immersion $f\in W^{2,2}_{loc}(D,\R^n)$ so that $f_k \rightharpoonup f$ weakly in $W^{2,2}_{loc}(D,\R^n)$. We denote by $u\in L^\infty_{loc}(D)$ the pointwise a.e. limit of $u_k$.

In a first step we show that $H_k e^{u_k}$ converges weakly to $He^u$ in $L^2_{loc}(D,\R^n)$. In order to see this we note that $ e ^ {u_k }H_k \rightharpoonup e ^ u H$ in $L^2_{loc}(D\backslash \{0\},\R^n)$ since $u_k$ converges pointwise almost everywhere to $u$ away from the origin.
Furthermore the hypothesis implies that $ e ^ {u _k } H _ k $ is uniformly bounded in $ L ^ 2_{loc}( D)$. Therefore, up to a subsequence, $ e ^{u_k } H_k $ weakly converges to some $\phi \in L ^ 2 (D)$. Hence, we must have the measure convergence
\begin{align*}
e^{u_k} H _k dx \rightharpoonup \phi dx = e ^ u H dx +  c \delta_0,  
\end{align*}    
where $\delta_0$ denotes the Dirac delta distribution supported on $\{0\}$. Since $\phi$ and $e^u H$ are in $L^2_{loc}(D,\R^n)$ this is only possible for $c=0$ and therefore $e^{u_k} H_k$ converges weakly to $e^u H$ in $L^2_{loc}(D,\R^n)$.

Next we let $K\subset \subset D$ and we conclude from the above
\begin{align*}
 \int_K |H_ke^{u_k} - He^{u}|^2 dx =& \int_{K }|H_{k } |^2 d \mu _ {g _ k}+\int _{ K} | H | ^ 2 d \mu _g-2\int_K H_k e^{u_k} H e^{u} dx\to 0.
\end{align*}
Hence $H_k e^{u_k}$ converges to $He^u$ strongly in $L^2_{loc}(D,\R^n)$. 

We note that the same argument implies that
\[
e^{-u_k} A_{k,ij} \to e^{-u} A_{ij}
\]
strongly in $L^2_{loc}(D)$ for all $1\le i,j \le 2$ and
\[
e^{-u_k} Df_k \to e^{-u} Df
\]
strongly in $L^2_{loc}(D)$.

Moreover we estimate
\begin{align*}
\| H _k e ^ { 2 u_k} - H e ^ { 2 u} \| _{ L ^ 2 ( K) } \leq \|H_k e ^ { 2u_ k} - e ^u H e ^ { u_k} + e ^ u H e ^ { u _ k } - H e ^{2u}    \|_{ L^2 ( K) } \\
\leq \| e ^ {u _k } \| _{ L ^ \infty(K)}   \| H _k e ^{u_k}- H e ^ u   \|_{ L ^ 2 ( K) } + \| ( e ^ {u_k} - e ^ u ) e ^ u H \| _{ L^ 2 ( K) }.       
\end{align*}
As $ \| e ^{u_k } \| _{L^ \infty(K) } \leq C$ and $ e ^ {u_k} H_k \rightarrow e ^ u H $ strongly in $ L^2(K,\R^n) $, the first term converges to $0$ as $k\to \infty$. Furthermore, the second term converges to zero by the dominated convergence theorem. Altogether this shows that
\begin{equation}\label{strong}
 H_k e^{2 u_k} \to He^{2u}\ \ \ \text{strongly}\,\ \text{in}\ \ L^2_{loc}(D).
\end{equation}
Finally, we recall that for conformal immersions $f_k$ and $f$ we have
\[
 \Delta (f_k-f)=H_k e^{2 u_k}-He^{2u}
\]
and standard elliptic theory and the Sobolev embedding theorem imply
\[
||f_k-f||_{L^\infty_{loc}(D,\R^n)}+||f_k-f||_{W^{2,2}_{loc}(D,\R^n)} \to 0.
\]
Next we want to improve this result in order to get the desired convergence in the weighted norm. In order to do this, we first note that $|f_k-f|(z)\le C|z|^{m+1}$ for all $z\in D$ (see for example the proof of Theorem 3.1 in \cite{Kuwert2010}). Hence we get for all $K\subset \subset D$ by using the dominated convergence theorem
\[
\int_K e^{-2u_k} |f_k-f|^2 \le C||f_k-f||_{L^\infty_{loc}(D,\R^n)}^{\frac{2}{m+1}} \to 0.
\]
Moreover, using that $e^{-u_k}Df_k$ converges strongly in $L^2_{loc}$ to $e^{-u} Df$, we have for all $K\subset \subset D$
\begin{align*}
\int_K e^{-2u_k} |D(f_k-f)|^2=& \int_K( 2+2e^{2u-2u_k}-2e^{-2u_k}\langle Df_k,Df \rangle)\\
=&\int_K\left(4-2\langle e^{-u_k}Df_k,e^{-u} Df\rangle \right)\\
&+2\int_K \left(e^{2u-2u_k}-1-(e^{u-u_k}-1)\langle e^{-u_k} Df_k, e^{-u}Df \rangle\right)\\
\to& 0.
\end{align*}
It remains to show that the second derivatives converge to zero in the weighted norm. In order to do this we recall the formula
\[
K_g=e^{-4u} (\langle A_{11},A_{22} \rangle-|A_{12}|^2)
\]
which is valid for any $W^{2,2}$-conformal immersion. Using this expression and the strong $L^2_{loc}(D)$-convergence of $e^{-u_k} A_{k,ij}$ to $e^{-u} A_{ij}$, for every $1\le i,j\le 2$, we get for every $K\subset \subset D$
\[
\int_K |K_{g_k}e^{2u_k}-K_g e^{2u}| \to 0.
\]
Since $u_k-u$ solves the equation
\[
-\Delta (u_k-u) =K_ge^{2u} -K_{g_k} e^{2u_k}
\]
we have for every smooth cut-off function $\eta \in C^\infty_c(D)$
\begin{align*}
\int_D  \eta^2 |D (u_k-u)|^2\le& C\int_D \eta^2 |K_{g_k}e^{2u_k}-K_g e^{2u}|+C\int_D |D\eta|^2 |u_k-u|^2\\
\to& 0,
\end{align*}
where we used the dominated convergence theorem and the above convergence result involving the Gauss curvatures. In particular we conclude that $Du_k$ converges to $Du$ strongly in $L^2_{loc}(D)$.

Finally, we also recall the formula
\[
e^{-2u} |D^2 f|^2= 4|Du|^2 +\sum_{i,j=1}^2 |A_{ij}|^2
\]
for every conformal immersion $f$. 

In particular, we conclude that $e^{-u_k} D^2f_k$ and $e^{-u} D^2f$ are uniformly bounded in $L^2_{loc}(D)$ and as above this shows that $e^{-u_k} D^2 f_k$ converges strongly in $L^2_{loc}(D)$ to $e^{-u}D^2 f$. Similarly to the estimate for the first derivatives we then finish the proof by noting that
\begin{align*}
\int_K e^{-2u_k} |D^2(f_k-f)|^2=&\int_K \big(e^{-2u_k}|D^2 f_k|^2+e^{2u-2u_k}e^{-2u} |D^2 f|^2\\
&-2e^{u-u_k} \langle e^{-u_k}D^2f_k,e^{-u}D^2 f\rangle \big)\\
\to& 0.
\end{align*}
\end{proof}

\section{Quantitative rigidity results}

\subsection{Quantitative rigidity for minimisers of $\mc W$}

The first main result of this paper is the following theorem.
\begin{thmr}
There exists $\delta_0>0$ such that for every $0<\delta<\delta_0$ and every immersion $f\in W^{2,2}_{\text{conf}}(\mbb S^2, \R^n)$ with conformal factor $u$, satisfying $\int_{\mbb S^2} |A|^2 d\mu \le 8\pi+\delta$, there exists a constant $C(\delta)$, with $C(\delta)\to 0$ as $\delta \to 0$, and a standard immersion $f_{\text{round}}\in W^{2,2}_{\text{conf}}(\mbb S^2,\R^n)$ of a round sphere such that 
\begin{align}
||f -f_{round}||_{W_u^{2,2}(\mbb S^2,\R^n)}  \le C(\delta).
\label{closeness1}
\end{align} 
\end{thmr}
\begin{remark}\label{compare}
In order to compare this result to Theorem \ref{thm1} we first note that for every closed surface $\Sigma$ and every immersion $f\in W^{2,2}\cap W^{1,\infty}(\Sigma ,\R^n)$ satisfying
\begin{align}
 \int_\Sigma |A|^2 d\mu <12\pi, \label{3pi} 
\end{align}
we have that $\Sigma$ is conformal to $\mbb S^2$. Indeed, using \eqref{Willeq} and the fact that $\mc W(f) \ge 4\pi$ for every immersion, we conclude that $\chi(\Sigma) > 1$.
This shows that $\Sigma$ must be homeomorphic to $\mbb S^2$. The result of Jost (see Theorem 3.1.1 in \cite{jost91}) yields the existence of a conformal parameterisation $h\in W^{1,2}(\mbb S^2,\Sigma)$. Finally, if $f\in C^\infty(\Sigma, \R^n)$ then $h\in C^\infty (\mbb S^2, \Sigma)$.
In particular, we conclude that for every smooth immersion $f\in C^\infty (\Sigma,\R^n)$ of a closed Riemann surface $\Sigma$ satisfying \eqref{3pi}, there exists a smooth conformal parameterisation $h:\mbb S^2 \to f(\Sigma)$, such that $h \in W^{2,2}_{\text{conf}}(\mbb S^2,\R^n)$.

Moreover, we note that \eqref{Willeq} implies that the bound $ \int_{\mbb S^2} |A|^2 d\mu \le 8\pi+\delta$ is equivalent to 
\[
\int_{\mbb S^2} |\circo A|^2 d\mu \le \frac{\delta}{2}.
\]
Altogether, we conclude that Theorem \ref{rigidity} is an extension of Theorem \ref{thm1} to arbitrary codimensions and non-smooth immersions. 
\end{remark}
\begin{proof}[Proof of Theorem \ref{rigidity}]
We prove this Theorem in two steps. First we show that under the assumptions of the Theorem, there exists a M\"obius transformation $\sigma:\R^n \to \R^n$ possibly consisting of a translation, a dilation and an inversion $I_{x_0}$ for some $x_0 \in \R^n$ with $B_1(x_0) \cap f(\mbb S^2) =\emptyset$ so that 
\begin{align}
||\sigma \circ f -f_{round}||_{W^{2,2}_v(\mbb S^2,\R^n)}  \le C(\delta),
\label{closeness2}
\end{align} 
where $v$ is the conformal factor of the conformal immersion $\sigma \circ f$. We show this result by contradiction. Hence we assume that there exists a sequence $\delta_k\to 0$ and a sequence of immersions $f_k\in W^{2,2}_{conf}(\mbb S^2, \R^n)$ such that 
\begin{align}
8\pi \le   \int_\Sigma |A_k|^2 d\mu_k \le 8\pi+\delta_k   \label{contra1}
\end{align}
but 
\begin{align}
||\sigma_k \circ f_k-f_{round}||_{W_{v_k}^{2,2}(\mbb S^2,\R^n)} \ge \e_0 ,\label{closenessa}
\end{align} 
for some number $\e_0>0$, for every standard round immersions $f_{round}$ of the round sphere and for all sequences of M\"obius transformations $\sigma_k:\R^n\to \R^n$ as described above. 

It follows from \eqref{gauss} that 
\[
4\pi\le  \mc W(f_k)=\frac14 \int_{\mbb S^2} |H_k|^2 d\mu_k \le \frac14  \int_{\mbb S^2} |A_k|^2 d\mu_k +2\pi \to 4\pi.
\]

We can apply Proposition \ref{prop_globalconv} in order to get the existence of an at most finite set $\mc S \subset \mbb S^2$ and a sequence of M\"obius transformations $\tilde{\sigma}_k$ so that $\tilde{f}_k:=\tilde{\sigma}_k \circ f_k$ converges weakly in $W^{2,2}_{loc}(\mbb S^2 \backslash \mc S,\R^n)$ to a map $f\in W^{2,2}_{conf,br}(\mbb S^2 ,\R^n)$. 
 
Next we claim that $\mc S=\emptyset$. We recall that $\mc S$ is defined by
\[
 \mc S=\{ p\in \mbb S^2:  \alpha(\{p\}) \ge \gamma_n\},
\]
where $\gamma_n$ is as in Theorem \ref{conv} and $\alpha$ is the limit of $\mu_{g_k} \llcorner |A_k|^2$ as Radon measures. For a point $p\in \mc S$ and a small radius $\varrho>0$ we then have that 
\[
\lim_{k\to \infty} \int_{B_\varrho(p)}|A_k|^2d\mu_k >\int_{B_\varrho(p)}|A|^2d\mu.
\]
But this clearly contradicts the fact that
\[
 8\pi = \lim_{k\to \infty} \int_{\mbb S^2}|\tilde{A}_k|^2 d\tilde{\mu}_k=\lim_{k\to \infty} \int_{\mbb S^2} |A_k|^2 d\mu  > \int_{\mbb S^2} |A|^2 d\mu \ge 8\pi 
\]
and hence we conclude that $\mc S=\emptyset$. This implies that $f\in W^{2,2}_{conf}(\mbb S^2,\R^n)$ with $\mc W(f)=4\pi$. Inverting $f(\mbb S^2)$ at any point on the surface gives a minimal surface with one end and $\int |A|^2 d\mu =0$ by \eqref{inversion2}, \eqref{inversion3}. Hence it must be a standard smooth immersion of the plane. Therefore $f$ itself must be a standard smooth immersion of a round sphere. Finally, we show that $\tilde{f}_k$ converges strongly to $f$ in the weighted norm $W_{v_k}^{2,2}(\mbb S^2,\R^n)$. In order to do this, we let $B$ be any disc in $\mbb S^2$ which is conformal to the unit disc $D\subset \R^2$ and we note that the above results imply that we can apply Proposition \ref{lem_strong} to conclude the strong convergence on $B$. Altogether, this yields a contradiction to \eqref{closenessa} and finishes the proof of the first step.

In the second step we show that there exists a constant $C>0$ such that 
\begin{align}
||f -f_{round}||_{W_u^{2,2}(\mbb S^2,\R^n)} \le& C||\sigma \circ f -\sigma \circ f_{round}||_{W^{2,2}_v(\mbb S^2,\R^n)} \nonumber \\
=& C||\sigma \circ f - f_{round}||_{W^{2,2}_v(\mbb S^2,\R^n)},\label{stereo}
\end{align}
where we used that a M\"obius transformation of a standard immersion of a round sphere is again a standard immersion of the round sphere.

Without loss of generality we can assume that $\sigma=I_0=:I$, $(I \circ f)\cap B_1(0) =\emptyset$ and $\mu_{I \circ f}(\mbb S^2)=1$. Defining $g=I\circ f$ and $g_0=I \circ f_{round}$, we see that the estimate \eqref{stereo} is equivalent to
\[
||I\circ g -I \circ g_0||_{W_u^{2,2}(\mbb S^2,\R^n)} \le C||g -g_0||_{W^{2,2}_v(\mbb S^2,\R^n)}
\]
with $v=u+2\log|g|$.

Using Lemma $1.1$ in \cite{simon93} and the uniform bounds on the area and Willmore energy of $g$, we conclude for all $x\in \mbb S^2$
\[
1\le |g|(x) \le C.
\]
Moreover, for $\delta$ small enough, we also get for all $x\in \mbb S^2$
\[
\frac12 \le |g_0|(x)\le 2C.
\]
Now standard calculations show the pointwise estimates
\begin{align*}
|I\circ g-I\circ g_0|=& |g-g_0| |g|^{-1} |g_0|^{-1},\\
|D(I\circ g)-D(I\circ g_0)|\le& c (|g|^{-2}+|g_0|^{-2})|Dg-Dg_0|\\
&+c(|g|^{-3}+|g_0|^{-3})(|Dg|+|Dg_0|)|g-g_0|\ \ \ \text{and}\\
|D^2(I\circ g)-D^2(I\circ g_0)|\le& c (|g|^{-2}+|g_0|^{-2})|D^2g-D^2g_0|\\
&+c(|g|^{-3}+|g_0|^{-3})\Big((|D^2g|+|D^2g_0|)|g-g_0|\\
&+(|Dg|+|Dg_0|)|Dg-Dg_0|\Big)\\
&+c(|g|^{-4}+|g_0|^{-4})(|Dg|^2+|Dg_0|^2)|g-g_0|.
\end{align*}
Combining all these estimates and using that $g,g_0\in W^{2,2}(\mbb S^2,\R^n)$ proves \eqref{stereo}.
\end{proof}

Next we extend the previous result to the case of surfaces of arbitrary genus if we assume that the Willmore conjecture is true. However, in the previous theorem we used strongly the fact that the conformal structure on the sphere is unique. For tori and higher genus surfaces, the moduli space is more complicated. The following theorem, requires that we restrict the immersions to a specific conformal class. For $n\ge 3$ and $p \in \N$ we define
\[
\beta^n_p:= \inf\{ \mc W(f): \,\ f:\Sigma\to \R^n\ \ \text{smooth, closed immersion},\ \ \text{genus}(\Sigma)=p\}.  
\]
Simon \cite{simon93} showed that $\beta^n_p$ is attained for $p=1$ and for $p\ge 2$ under the additional assumption
\[
 \beta^n_p< \min\{ 4\pi+\sum_i (\beta^n_{p_i} -4\pi):\,\ \sum_i p_i=p,\ \ 1\le p_i<p\}=: \omega^n_p.
\]
Later, Bauer and Kuwert \cite{B-K}, showed that the above inequality is always satisfied and hence there exists a minimising Willmore surface for every genus. The (generalised) Willmore conjecture now states that 
this minimiser is unique modulo M\"obius transformations. Moreover, it is conjectured that the minimisers are the stereographic images of the minimal surfaces in $\mbb S^n$ found by Lawson \cite{lawson70}. 

Finally we remark that Pinkall \cite{kuehnel86} and Kusner \cite{kusner89} showed that the Willmore energy of the stereographic images of Lawson's minimal surfaces is strictly less than $8\pi$. In particular, by combining the above results, we have for every $p\in \N$
\begin{align}
 \beta^n_p <\min\{ 8\pi, \omega^n_p\}. \label{estbeta} 
\end{align}
In the following we assume that the Willmore conjecture is true, that is for every $p\in \N$ there exists a unique (up to M\"obius transformations) Willmore immersion $f_p:\Sigma_p \to \R^n$ with $\text{genus}(\Sigma_p)=p$ and $\mc W(f_p)=\beta^n_p$.
\begin{theorem}\label{genus}
Let $p\in \N$ and let $\Sigma$ be a smooth Riemann surface of genus $p$ for which there exists a smooth conformal diffeomorphism $\phi: \Sigma_p \to \Sigma$. There exists $\delta_p>0$ such that for every $0<\delta <\delta_p$ and every immersion $f\in W^{2,2}_{conf}(\Sigma,\R^n)$, satisfying $\mc W(f) \le \beta^n_p+\delta$, there exists a M\"obius transformation $\sigma:\R^n \to \R^n$, a constant $C(\delta)$, with $C(\delta)\to 0$ as $\delta \to 0$, and a standard immersion $f_p\in W^{2,2}_{conf}(\Sigma_p,\R^n)$ of the assumed minimiser $\Sigma_p$ of $\mc W$ of genus $p$ such that
\begin{align}
||\sigma \circ f\circ \phi -f_p||_{W_u^{2,2}(\Sigma_p,\R^n)} \le C(\delta),\label{genus1} 
\end{align}
where $u$ is the conformal factor of the conformal immersion $\sigma \circ f\circ \phi$.
\end{theorem}
\begin{proof}
Since the proof is very similar to the first part of the proof of Theorem \ref{rigidity} we only sketch the main ideas. We argue by contradiction, that is we consider sequences $\delta_k\to 0$, $f_k \in W^{2,2}_{conf}(\Sigma_k,\R^n)$ and smooth conformal diffeomorphisms $\phi_k:\Sigma_p \to \Sigma_k$ such that 
\begin{align*}
 \beta^n_p\le \mc W(f_k) \le \beta^n_p+\delta_k \to \beta^n_p 
\end{align*}
and
\begin{align*}
 ||\sigma_k \circ f_k\circ \phi_k -f_p||_{W_{u_k}^{2,2}(\Sigma_p,\R^n)} \ge \e_0
\end{align*}
for some $\e_0>0$, all sequences of M\"obius transformations $\sigma_k:\R^n\to \R^n$ as above and all standard immersions of the assumed minimiser $f_p$ as above. 

Using \eqref{estbeta}, it follows from Theorem $5.3$ and Theorem $5.5$ in \cite{Kuwert2010}, respectively Theorem I.$1$ in \cite{Rivi`ere2011}, that $f_k\circ \phi_k$ converges in moduli space. Hence, we can apply Proposition \ref{prop_globalconv} in order to conclude that there exists a sequence of M\"obius transformations $\tilde{\sigma}_k:\R^n\to \R^n$ as in the statement of the Theorem, so that $\tilde{f}_k:=\tilde{\sigma}_k\circ f_k \circ \phi_k$ converges weakly in $W^{2,2}$ to a conformal immersion $f$ away from an at most finite set $\mc S$. 
Since $\mc W(f)=\beta^n_p$ and we assumed the validity of the Willmore conjecture we conclude that $f$ must be a standard Willmore minimising immersion of $\Sigma_p$. The rest of the argument now follows as in the first part of the proof of Theorem \ref{rigidity}.  
\end{proof}
\begin{remark}
In a recent paper \cite{Marques2012}, Marques and Neves prove the Willmore conjecture for tori in codimension one. Together with the above theorem, this yields a quantitative rigidity result for conformal immersions of the Clifford torus in $\R^3$ in its conformal class.

We also note that in two interesting papers, Ndiaye and Sch\"atzle \cite{NS1}, \cite{NS2}, showed that the Clifford torus minimises the Willmore energy in an open neighbourhood in moduli space of its conformal class.
\end{remark}

\subsection{Quantitative rigidity for spheres with a double point or a branch point}

The proof of Theorem \ref{rigidity} above is very flexible and in this subsection we will adapt it to prove similar statements corresponding to higher energy levels.
\subsubsection{Codimension one}
Let us consider an immersion $f\in W^{2,2}_{conf}(\mbb S^2, \R^3)$, where $ f$ has a point of multiplicity two $x \in f ( \Sigma)$ such that $ f^{-1} ( x) = \{ p_1, p_2\}$ and $p_1\not= p_2$. It follows from \eqref{inversion2} that $\mc{W}( f)\geq 8 \pi$. Furthermore, by the Gauss-Bonnet formula, we have that
\begin{align*}
 \int_{\mbb S^2 } | A|^{2} d\mu &= -2\int_{\mbb S^2} Kd \mu +  \int_{\mbb S^2} H^2 d \mu \geq - 8 \pi + 32 \pi = 24 \pi.   
\end{align*}   
We remark that every $f\in W^{2,2}_{\text{conf}}(\mbb S^2, \R^3)$ with at least one point of multiplicity two such that $ f^{-1} ( x) = \{ p_1, p_2\}$ and which satisfies $\mc W(f)= 8\pi$ is a standard immersion of an inverted catenoid. Indeed, inverting the surface at the double point gives us a minimal surface with two ends, which is $C^\infty$ away from the inverted points. Hence this shows that the immersion must be an inversion of the catenoid by \eqref{inversion2}, \eqref{inversion3} and Theorem \ref{thm_hoss}.

In the following theorem we show that there is a rigidity result associated to such a lower energy bound in codimension one. 
\begin{thmr2}
There exists $\delta_{cat}=\delta_0>0$ such that for every $0<\delta<\delta_0$ and every immersion $f\in W^{2,2}_{\text{conf}}(\mbb S^2, \R^3)$ with conformal factor $u$ which has exactly one point of multiplicity two $ x \in f ( \mbb{S}^2)$ with $ f^{-1} (x) = \{ p_{1}, p_{2} \}$ and which satisfies $24 \pi \leq  \int_{\mbb S^2} |A|^2 d\mu \le 24\pi+\delta$, there exists a M\"obius transformation $\sigma:\R^n \to \R^n$, a reparameterisation $\phi:\mbb S^2 \to \mbb S^2$, a
constant $C(\delta)$, with $C(\delta)\to 0$ as $\delta \to 0$, and a standard immersion $f_{\text{cat}}\in W^{2,2}_{\text{conf}}(\mbb S^2,\R^3)$ of an inverted catenoid with
\begin{align}
||\sigma \circ f \circ \phi-f_{cat}||_{W_u^{2,2}(\mbb S^2,\R^3)} \le C(\delta)\sqrt{\mu_g(\mbb S^2)}. \label{closea}
\end{align} 
\end{thmr2}
\begin{proof} 
We prove this estimate by contradiction. Let us assume that there exists a sequence $\delta_k \rightarrow 0$ and a sequence of immersions $ f _k \in W ^{2,2} _{ conf} ( \mbb S^2, \R^ 3 )$ which have a point of multiplicity two $ x_k \in f_k ( \mbb{S}^2)$ such that $ f_k^{-1} (x_k) = \{ p_{1,k}, p_{2,k} \}$ and which satisfies 
\begin{align}\label{contrad}
24 \pi \leq  \int_{\mbb S^2} |A_k|^2 d\mu_k \le 24\pi+\delta_k .
\end{align} 
but 
\begin{align}
||\sigma_k \circ f_k \circ \phi_k-f_{cat}||_{W^{2,2}_{v_k}(\mbb S^2,\R^n)} \ge \e_0,\label{cat_close}
\end{align} 
for some $ \e_0 > 0$, all sequences of M\"obius transformations $\sigma_k$, all sequences of reparameterisations $\phi_k$ and all standard immersions $f_{cat}$ of an inverted catenoid. Here $v_k$ is the conformal factor of $\sigma_k \circ f_k \circ \phi_k$. Using the Gauss-Bonnet formula once more, we conclude 
\begin{align*}
8\pi \leq  \mc W( f_k ) = \frac{1}{4} \int_{ \mbb S^2} | H_k|^2 d \mu _ k \leq \frac{1}{4} \int_{ \mbb S^2} |A_k|^ 2 d \mu _ k + 2 \pi \rightarrow 8 \pi. 
\end{align*}

Applying Proposition \ref{prop_globalconv} gives the existence of a finite set $\mc{S} \subset \mathbb{S} ^ 2$ and a sequence of M\"obius transformations $\tilde{\sigma}_k:\R^n\to \R^n$ so that $\tilde{f}_k:= \tilde{\sigma}_k \circ f _k$ converges weakly in $ W^{2,2}_{loc} ( \mbb S^2 \backslash \mc {S}, \R^3 )$ to $\tilde f\in W^{2,2}_{loc}(\mbb S^2 \backslash \mc S,\R^3)$. Note that $\tilde f_k$ still has a double point $y_k\in \tilde f_k(\mbb S^2)$.

Next we show that there exists another sequence of M\"obius transformations $ \sigma_k$ and a sequence of reparameterisations $\phi_k$, so that $\hat f_k:=\hat \sigma_k \circ \tilde f _ k \circ \phi_k $ converges weakly in $W^{2,2}_{loc}$ away from at most finitely many points to a possibly branched conformal immersion $ f$ which has a point of multiplicity two. By choosing an appropriate reparameterisation we can assume that $ \tilde f _k^{-1} (y_k)= \{ N,S \} $. Inverting the image surface $\tilde f_k(\mbb S^2)$ at $y_k$ and noting that $\mbb S^2 \backslash \{N,S\}$ is conformal to a cylinder $\R \times \mbb S^1$, we get a sequence $h_k$ of $W^{2,2}$-conformal immersions from the cylinder to $\R^3$ with two ends.
Now, for each $k$, we consider the length of the image of the curves $C_v:=\{v\} \times \mbb S^1$ and we note that it follows from the isoperimetric inequality in Theorem 5.2 in \cite{li91} that the length of $h_k(C_v)$ tends to infinity as $v\to \pm \infty$. Moreover, using Corollary 4.2.5 in \cite{Muller1995}, this implies that also the extrinsic diameter of $h_k(C_v)$ tends to infinity as $v\to \pm \infty$. Hence we conclude that $\inf_{v\in \R} \diam ( h_k ( C _v ) ) >0$. After performing a translation and dilation, we may therefore assume that $\diam ( h_k ( C _v ) )\ge \diam (h_k(C_{v_k})) =1 $ and $ 0 \in h_k(C_{v_k} ) $. Next, we argue as in the proof of Theorem $5.3$ in \cite{Kuwert2010}, and we conclude that there exists an inversion $I_{x_0}$ so that $I_{x_0} \circ h_k$ converges locally weakly in $W^{2,2}(\mbb S^2 \backslash \mc S,\R^n)$ to a branched conformal immersion $f$ with a double point.
 
These arguments, combined with \eqref{inversion2}, yield
\begin{align*}
\mc W(f)  = 8\pi.
\end{align*}
Next we will show that $ \mc {S} = \emptyset$. Recall that $\mc{S}$ is defined by 
\begin{align*}
\mc {S} = \{ p \in \Sigma: \alpha ( \{p\} ) \geq 8 \pi \} 
\end{align*}  
where $ \alpha $ is the limit of $ \mu_{\hat{g}_k } \llcorner |\hat{A}_k|^2 $ as a Radon measure. Now applying the Gauss-Bonnet formula \eqref{gauss-bonnet} we get that  
\begin{align}
\int_{\mbb S^2} |A|^2 d \mu  &=  4  \mc W(f) - 2 \int_{\mbb S^2} K d \mu \nonumber\\
&= 32 \pi - 4 \pi\left ( \chi(\mbb S^2) + \sum_{p_i \in \mc S } m ( p_i) \right) \nonumber\\
& = 24\pi - 4 \pi\sum_{ p _ i \in \mc {S} } m( p_i ) \label{contra}.  
\end{align}

Since $f$ has at least one point of multiplicity two we conclude from \eqref{inversion3} that $ \int_{\mbb S^2} |A|^2 d\mu \geq 16\pi$. Combining the last two estimates shows that
 \begin{align*}
\sum_{ p _ i \in \mc {S} } m ( p_i ) \leq 2.
\end{align*} 
Note that by Theorem 3.1 in \cite{Kuwert2010} we have $m(p)\ge 1$ for all $p\in \mc S$. We assume that $\mc S \not= \emptyset$. Then we conclude from the above estimate that $\mc S$ consists of either one point $p$ with $m(p)=2$, or one point $p$ with $m(p)=1$ or two points $p,q$ with $m(p)=m(q)=1$. 

The first possibility is ruled out by \eqref{inversion2}, since we know that $\mc W(f)= 8\pi$. In order to rule out the second case, we note that \eqref{contra} implies in this situation
\[
 \int_{\mbb S^2} |A|^2 d \mu=20 \pi.
\]
Using the definition of $\mc S$ and recalling that  $ \lim_{ k\rightarrow \infty }\int_{\mbb S^2}  | \tilde{A}_ k | ^ 2 d \tilde{\mu} _ k \rightarrow 24 \pi$ we get the contradiction
\begin{align*}
24 \pi = \lim_{ k\rightarrow \infty }\int_{\mbb S^2} | \hat{A}_ k | ^ 2 d \hat{\mu} _ k \ge \int_{\mbb S^2} |A|^ 2 d\mu+ \alpha ( \{ p\} ) \ge 28 \pi.
\end{align*}
We argue similarly in the third case in order to get the contradiction
\begin{align*}
24 \pi = \lim_{ k\rightarrow \infty } \int_{\mbb S^2} | \hat{A}_ k | ^ 2 d \hat{\mu} _ k \ge \int_{\mbb S^2} |A|^ 2 d\mu + \alpha ( \{ p\} ) + \alpha ( \{ q \} ) \geq 16\pi+8\pi+8\pi=  32 \pi.
\end{align*}
Hence we must have that $ \mc {S}= \emptyset$. In particular, this implies that the point of multiplicity two $0\in f(\mbb S^2)$ corresponds to two points $\{p,q\} =f^{-1}(0)$ with $p\not= q$ and $m(p)=m(q)=0$. 

Combining all the above facts we use \eqref{inversion3} and Theorem \ref{thm_hoss} in order to conclude that $f$ must be a standard immersion of an inverted catenoid. 
Furthermore we have that $ \int_{\mbb S^2} |A| ^ 2 d \mu = 24 \pi$. This shows that 
\begin{align*}
\int_{\mbb S^2} |\hat{A}_k | ^ 2 d \hat{\mu} _k \rightarrow \int_{\mbb S^2} |A| ^2 d \mu = 24 \pi
\end{align*}  
and hence we can then apply Proposition \ref{lem_strong} to conclude that the immersions $\hat{f}_k$ do not only converge weakly to an immersion of the inverted catenoid, but in fact they converge strongly in the weighted $W^{2,2}(\mbb S^2,\R^3)$-norm, which yields a contradiction.

Since the second step in the proof of Theorem \ref{rigidity} directly carries over to this situation we finish the proof of the Theorem.
\end{proof} 

We prove a similar theorem for an inverted Enneper's minimal surface. 
Let us consider an immersion $ f\in W ^{2,2} _{ conf,br }( \mbb{S}^{2} , \R^3) $ which has exactly one branch point $p$ of branch order $ m(p)=2 $. For such a surface we have that $\mc { W}( f) \geq 12 \pi$ by \eqref{inversion2}. The existence of a point $p$ with $m(p)=2$ also implies the bound $ \int_{ \mbb S^2} |A|^2 d \mu \geq 24 \pi$ by \eqref{inversion3}. However, applying the Gauss-Bonnet formula \eqref{gauss-bonnet}, we even get  
\begin{align*}
\int_ { \mbb S^2} |A| ^ 2 d\mu &= 4 \mc { W} ( f) - 2 \int _ { \mbb S^2} K d \mu \\
& \geq 48 \pi - 4 \pi ( \chi(\mbb{S}^2 ) + m (p))\\
&=  32\pi.
\end{align*}  
Furthermore note that a $W^{2,2}$ immersion of the sphere that has exactly one branch point of branch order two and satisfies $ \mc {W} = 12 \pi$, after inversion at the branch point, is a minimal surface with an end of multiplicity three and hence is $C^\infty$ away from the branch point. This shows that the immersion is an inversion of Enneper's minimal surface by \eqref{inversion2}, \eqref{inversion3} and Theorem \ref{thm_hoss}. 

\begin{thmr3}
There exists $\delta_{Enn}=\delta_0>0$ such that for every $0<\delta<\delta_0$ and every immersion $f\in W^{2,2}_{\text{conf}, br }(\mbb S^2, \R^3)$ with conformal factor $u$ where $ f $ has exactly one branch point $p\in \mbb S^2$ of branch order $m(p) = 2 $, and which satisfies $32 \pi \leq  \int_{\mbb S^2} |A|^2 d\mu \le 32\pi+\delta$ then there exists a M\"obius transformation $\sigma:\R^n \to \R^n$, a reparameterisation $\phi:\mbb S^2 \to \mbb S^2$, a
constant $C(\delta)$, with $C(\delta)\to 0$ as $\delta \to 0$, and a standard immersion $f_{\text{Enn}}\in W^{2,2}_{conf,br}(\mbb S^2,\R^3)$ of an inverted Enneper's minimal surface with
\begin{align}
||\sigma \circ f \circ \phi-f_{Enn}||_{W_u^{2,2}(\mbb S^2,\R^3)} \le C(\delta)\sqrt{\mu_{g}(\mbb S^2)}. \label{Enn_closea}
\end{align} 
\end{thmr3}
\begin{proof} 
We prove the theorem by contradiction. Let us assume that there exists a sequence $\delta_k \rightarrow 0$ and a sequence of immersions $ f _k \in W ^{2,2} _{ conf, br } ( \mbb S^2, \R^ 3 )$ such that $f_k$ has exactly one branch point $p_k$ of branch order $m(p_k)=2$ and satisfies 
\begin{align}\label{contrad3}
32 \pi \leq  \int_\Sigma |A_k|^2 d\mu_k \le 32\pi+\delta_k 
\end{align} 
but 
\begin{align}
||\sigma_k \circ f_k \circ \phi_k-f_{Enn}||_{W_{v_k}^{2,2}(S^2,\R ^3)} \ge \e_0 ,\label{Enn_close2}
\end{align} 
for some $ \e_0 > 0$, all sequences of M\"obius transformations $\sigma_k:\R^n\to \R^n$, all sequences of reparameterisations $\phi_k$ and all standard immersions $f_{Enn}$ of an inverted Enneper`s minimal surface. 

Furthermore, as a consequence of the Gauss-Bonnet formula \eqref{gauss-bonnet} we have that 
\begin{align*}
12\pi \leq  \mc W( f_k ) = \frac{1}{4} \int_{ \mbb S^2} | H_k|^2 d \mu _ k \leq \frac{1}{4} \int_{ \mbb S^2} |A_k|^ 2 d \mu _ k + \frac12 \int_{\mbb S^2} K_k d\mu_k  \rightarrow 12 \pi. 
\end{align*}

Hence we can apply Theorem \ref{branchpreserv} to get that there exists a sequence of M\"obius transforms  $ \tilde{\sigma} _k : \R^ 3 \rightarrow \R ^3$, a sequence of reparamterisations $\tilde \phi_k$ and a finite set $ \mc {S}\subset \Sigma$ so that $\tilde{f}_k:= \tilde{\sigma}_k \circ f _k \circ \tilde \phi_k$ converges weakly in $ W^{2,2} ( \mbb S^2 \backslash (\{p\}\cup\mc {S}), \R^3 ) $ (where $p=\lim_{k\to \infty} p_k$) to a branched conformal immersion $ f \in W^{2,2}_{conf,loc} ( \mbb S^2\backslash (\{p\}\cup\mc{S}) , \R ^ 3)$ which has at least one branch point of branch order $2$. 
From the weak lower semi-continuity and the conformal invariance of the energy we conclude $\mc W(f)\le 12\pi$. On the other hand, using the fact that we have at least one point of multiplicity three,  we get
\begin{align*}
\mc W(f) = 12\pi.
\end{align*}
 
Next we apply the Gauss-Bonnet formula \eqref{gauss-bonnet} to get that 
\begin{align*}
 \int_{\mbb S^2}  |A|^2 d\mu &= 4 \mc { W}( f) - 2 \int_{\mbb S^2} K d\mu \\
 & = 48 \pi - 4 \pi \bigg( \chi( \mbb S^2) + m(p) + \sum_{ p_i \in \mc S } m (p_i)\bigg) \\
 & = 32 \pi - 4 \pi \sum_{ p_ i \in \mc S } m (p_i). 
\end{align*}
As $m(p)=2 $ we have that $ \int_{\mbb S^2} |A|^ 2 d\mu \geq 24 \pi$ by \eqref{inversion3}. So we see that either $\mc S=\emptyset$ or $  \sum _{ p_i \in \mc {S} } m (p_i) = 1,2$. Hence we have the following possibilities
\begin{enumerate} 
\item $\mc S=\emptyset$,
\item $\mc {S} = \{ p_1, p_ 2 \}$ with $ m(p_1) = m(p_2) =1 $,
\item $\mc {S}=  \{p_1\} $ with $ m(p_1) = 1 $ or 
\item  $ \mc {S} = \{p_1\} $ with $ m(p_1) =2 $.
\end{enumerate}  
\textbf{Case 1)}\newline
In this case $f$ has exactly one branch point $p$ with $m(p)=2$, $\mc W(f)=12\pi$, $\int_{\mbb S^2} |A|^2 d\mu=32\pi$ and hence, using \eqref{inversion2}, \eqref{inversion3} and Theorem \ref{thm_hoss}, we conclude that $f$ must be a standard immersion of an inversion of Enneper`s minimal surface.
As $32\pi =\int_{\mbb S^2} |A|^2 d\mu =\lim_{k\to \infty} \int_{\mbb S^2} |\tilde{A}_k|^2 d\tilde{\mu}_k$ we conclude that $p$ can not be an energy concentration point and therefore the results of section $3$ yield a contradiction to \eqref{Enn_close2}.
\newline
\textbf{Case 2)}\newline
Note that $ \mc { S} = \{ q \mid \alpha ( \{q\} ) \geq 8 \pi\}$ where $ \alpha $ is the limit measure of $ \tilde{\mu}_k \llcorner |\tilde{A}_k|^2 $ and hence we get the contradiction 
\begin{align*}
 32 \pi &= \lim _{ k \rightarrow \infty } \int_{\mbb S^2} | \tilde{A}_k |^2d\tilde{\mu}_k \geq \int_{\mbb S^2} |A| ^ 2d\mu  + \alpha ( p_1)  + \alpha (p_2) \geq 24 \pi + 16 \pi. 
\end{align*} 
This shows that case 2) can not occur.\newline
\textbf{Case 3)}\newline
As in case 2) we obtain the following contradiction
\begin{align*}
 32 \pi &= \lim _{ k \rightarrow \infty } \int_{\mbb S^2} | \tilde{A}_k |^2d\tilde{\mu}_k \geq \int_{\mbb S^2} |A| ^ 2d\mu  + \alpha ( p_1) \geq 28 \pi + 8 \pi. 
\end{align*}
 \newline
\textbf{Case 4)}\newline
In this case we have a single branch point $p_1\in \mc S$ of branch order $m(p_1)=2$. In particular note that 
\begin{align*}
 32 \pi = \lim_{ k \rightarrow \infty } \int_{\mbb S^2} |\tilde{A}_k|^2 d\tilde{\mu}_k  & \geq \int_{\mbb S^2}  |A|^2 d\mu + \alpha ( p_1)\ge 24\pi +\alpha(p_1)  \ge 32 \pi
\end{align*}
and hence we have that $ \alpha ( p_1) = 8 \pi$. Therefore we conclude that there exists a disc $ B _r ( p_1) \subset \mbb S^2 $ such that 
\begin{align*}
 \lim_{k \rightarrow \infty } \int _{ B _r (p_1)} |\tilde{A}_k| ^ 2 d\tilde{\mu}_k = 8\pi.
\end{align*} 
Hence the surface loses exactly $ 8 \pi$ of total curvature at the point $p_1\in \mbb S^2$. Using local conformal coordinates we can actually assume that 
\begin{align*}
 \lim_{k \rightarrow \infty } \int _{ D _1 (p_1)} |\tilde{A}_k| ^ 2 d\tilde{\mu}_k = 8\pi,
\end{align*} 
where $D_1(p_1) \subset \R^2$. Using a convergence result of Li, Luo and Tang \cite{li}, the second author showed in \cite{Nguyen2011b}, Theorem $1.1$, that in this situation there exists a sequence of M\"obius transformations $\hat{\sigma}_k$, consisting of rescalings and translations, and a sequence of reparameterisations $\hat \phi_k$, so that $\hat{f}_k:=\hat{\sigma}_k \circ \tilde{f} _k \circ \hat \phi_k  $ converges locally weakly in $W^{2,2}_{loc}(\R^2,\R^3)$ to Enneper's minimal surface. More precisely, the sequence of M\"obius transformations is obtained by performing a suitable blow-up around a sequence of points $p_{1,k}\in D_1(p_1)$ with a sequence of radii $r_k$, for which one has
\begin{align*}
\int_{D_{r_k}(p_{1,k})} |\tilde{A}_k|^2 d\tilde{\mu}_k =8\pi-\tau=\sup_{x\in D_1(p_1)} \int_{D_{r_k}(x)} |\tilde{A}_k|^2 d\tilde{\mu}_k,
\end{align*}
where $\tau>0$ is a given number.

By a result of Kuwert-Li \cite{Kuwert2010} (see also \cite{Kuwert2012}) there exists a ball $B_1(x_0)\subset \R^3$ such that $\hat{f}_k(\R^2)  \cap B_1(x_0) =\emptyset$ for all $k\in \N$. Therefore we conclude that the sequence $\overline{f}_k:=I_{x_0} \circ \hat{f}_k$ converges weakly in $W^{2,2}(\mbb S^2 \backslash \mc S,\R^3)$ to a standard inversion of Enneper's minimal surface. Note that there are no energy concentration points for the sequence $ \overline{f}_k$ by our initial convergence result, the blow-up procedure and the fact that
\begin{align*}
32 \pi =& \int_{\mbb S^2} |\overline A|^2 d\overline \mu\\
\le& \liminf_{k\rightarrow\infty} \int_{\mbb S^2} |\overline A_k|^2 d \overline \mu_k\\
\le& \liminf_{k\rightarrow\infty} \int_{\mbb S^2} |\hat{A}_k|^2 d\hat{\mu}_k+24\pi \\
\le& 32\pi.
\end{align*}

From the results of section $3$ we conclude that $I_{x_0} \circ \hat{f}_k$ converges strongly in  $W^{2,2}(\mbb S^2,\R^3)$  to a standard inversion of Enneper's minimal surface. But this contradicts \eqref{Enn_close2} and completes the proof of the Theorem.
\end{proof}
\subsubsection{Higher codimensions}
For higher codimensions we have a result similar to Theorem \ref{rigidity3} but at a lower energy level and with Enneper's minimal surface replaced by Chen's minimal graph.
\begin{theorem}\label{rigidity4}
There exists $\delta_{Chen}=\delta_0>0$ such that for every $0<\delta<\delta_0$ and every immersion $f\in W^{2,2}_{\text{conf},br}(\mbb S^2, \R^n)$ with conformal factor $u$, $ n \geq 4$, where $ f $ has exactly one branch point of branch order $m(p) = 1 $, $ p \in   \mbb{S}^2$ and satisfies $20 \pi \leq  \int_{\mbb S^2} |A|^2 d\mu \le 20\pi+\delta$, then there exists a M\"obius transformation $\sigma:\R^n\to \R^n$, a reparameterisation $\phi:\mbb S^2 \to \mbb S^2$, a
constant $C(\delta)$, with $C(\delta)\to 0$ as $\delta \to 0$, and a standard immersion $f_{\text{Chen}}\in W^{2,2}_{\text{conf}}(\mbb S^2,\R^n)$ of an inverted Chen's minimal graph with
\begin{align}
||\sigma \circ f \circ \phi-f_{Chen}||_{W_u^{2,2}(\mbb S^2,\R^n)} \le C(\delta). \label{chen_close}
\end{align} 
\end{theorem}
\begin{proof} 
Again, we show this result by contradiction and we assume that there exists a sequence $\delta_k \rightarrow 0$, a sequence of immersions $ f _k \in W ^{2,2} _{ conf, br } ( \mbb S^2, \R^ n )$ such that each $f_k$ has a branch point $p_k$ of branch order $m(p_k)=1$ and satisfies 
\begin{align}\label{contrad4}
20 \pi \leq  \int_{\mbb S^2} |A_k|^2 d\mu < 20\pi+\delta_k .
\end{align} 
but 
\begin{align}
||\sigma_k \circ f_k \circ \phi_k-f_{Chen}||_{W_{v_k}^{2,2}(\mbb S^2,\R ^3)} \ge \e_0,\label{chen_close3}
\end{align} 
for some $ \e_0 > 0$, every sequence of M\"obius transformations $\sigma_k:\R^n\to \R^n$, ever sequence of reparameterisations $\phi_k$ and every standard immersion $f_{Chen}$ of an inverted Chen`s minimal graph. Furthermore, as a consequence of the Gauss-Bonnet formula we have that 
\begin{align*}
8\pi \leq  \mc W( f_k ) = \frac{1}{4} \int_{ \mbb S^2} | H_k|^2 d \mu _ k \leq \frac{1}{4} \int_{ \mbb S^2} |A_k|^ 2 d \mu _ k + \frac12 \int_{\mbb S^2} K_k d\mu_k \rightarrow 8 \pi
\end{align*}
where we used that $\frac12 \int_{\mbb S^2} K_k d\mu_k = \pi (\chi(\mbb S^2)+m(p_k))=3\pi$ by \eqref{gauss-bonnet}.

By Theorem \ref{branchpreserv} there exists a sequence of M\"obius transforms  $ \tilde{\sigma} _k : \R^ n \rightarrow \R ^n$, a sequence of reparameterisations $\tilde \phi_k$ and a finite set $ \mc {S}\subset \mbb S^2$ so that $\tilde{f}_k:= \tilde{\sigma}_k \circ f _k \circ \tilde \phi_k$ converges weakly in $ W^{2,2} ( \mbb S^2 \backslash (\{p\}\cup\mc {S}), \R^n ) $ to a conformal immersion $ f \in W^{2,2}_{loc} ( \mbb S^2 \backslash (\{p\}\cup\mc{S}) , \R ^ n)$ with at least one branch point of order $1$. From the weak lower semi-continuity, the conformal invariance of the energy and the fact that $f$ has a point of multiplicity two, we get that 
\begin{align*}
\mc W(f) = 8\pi.
\end{align*}
We let $p=\lim_{k\to \infty} p_k$ be the branch point of branch order $m(p)=1$ of $f$.

Furthermore, note that $ f$ completes as a branched immersion over $ \mbb S ^ 2 \backslash (\{p \} \cup \mc { S}) $ and we have the estimate
\begin{align*}
&\int_{\mbb S^2} |A| ^ 2 d \mu \leq 20 \pi .  
\end{align*} 
Hence, by the Gauss-Bonnet theorem \eqref{gauss-bonnet}, we have that 
\begin{align*}
\int_{\mbb S^2} |A|^2 d\mu &= 4 \mc W(f) - 2 \int_{\mbb S^2} K d \mu \\
&= 32 \pi - 4 \pi \bigg( 2 + m(p)+ \sum _{ p_i \in  \mc {S} } m ( p _ i ) \bigg)\\
& = 2 0 \pi -  4 \pi  \sum _{ p_i \in   \mc {S}} m ( p _ i ) . 
\end{align*}
As $16 \pi  \le \int_{\mbb S^2} |A|^2 d\mu$ by \eqref{inversion3}, we must have that 
\begin{align*}  
\sum _{ p_i \in   \mc {S}  } m ( p _ i ) \le 1.
\end{align*} 
Therefore we have two cases, either $\mc S =\emptyset$ or $ \mc S=\{p_1\} $, $p_1\not= p$ and $m(p_1) = 1$. \newline
\textbf{Case 1)}  $\mc S =\emptyset$ \newline
In this case we have that $ \int_{\mbb S^2} |A| ^ 2 d\mu = 20 \pi$, $\mc W (f)= 8\pi$ and $m(p)=1$. Inverting the surface at $f(p)$ gives a minimal surface with total curvature equal to $4 \pi$, which then must be Chen's minimal graph by Theorem \ref{thm_chen}. Arguing as in Case 1) of the proof of Theorem \ref{rigidity3} we get a contradiction to \eqref{chen_close3}.
\newline
\textbf{Case 2)}  $\mc S=\{ p_1\}$ with $m(p_1)=1$ 
\newline
Arguing similar to Case 4) in the proof of Theorem \ref{rigidity3} we conclude that $\alpha (p_1)=4\pi$. Hence we can apply once more Theorem $1.1$ from \cite{Nguyen2011b} and conclude that after performing a blow-up, the sequence of immersions converges to Chen's minimal graph.

Then we repeat the same argument as in Case 4) of the proof of Theorem \ref{rigidity3} in order to get that after composing the sequence of blow-up immersions with a suitable inversion, the new sequence converges strongly in $W^{2,2}(\mbb S^2, \R^n)$ to a standard inversion of Chen's minimal graph, contradicting \eqref{chen_close3}.
\end{proof}

\section{Quantitative rigidity for complete surfaces with finite total curvature} 

In this section we use the rigidity results from section $4$ in order to prove quantitative rigidity results associated to Theorem \ref{1}. We start with the case of codimension one.
\begin{theorem}
Let $ f : \Sigma \hookrightarrow \R ^ 3  $ be a complete, connected, non-compact surface immersed into $ \R ^ 3 $. Assume that  
\begin{align*}
\int_{\Sigma} |A| ^ 2 d\mu \leq 8 \pi + \delta,
\end{align*}
where $ \delta < \delta_{0} =\min\{\delta_{cat}, \delta_{Enn},8\pi, 4(\beta^3_1-4\pi) \}$ and $\delta_{cat}, \delta_{Enn}$ are given by Theorem \ref{rigidity2} and Theorem \ref{rigidity3}. Then either $f$ is a conformal immersion of a plane embedded in $ \R^ 3 $ or, after composition with an inversion $I_x$, $x\notin f(\Sigma)$, the surface $I_x \circ f$ is $ W^{2,2}$ close to an inversion of a catenoid or Enneper's minimal surface, modulo M\"obius transformations and reparameterisations.   
\end{theorem}
\begin{remark}
Note that by the recent proof of the Willmore conjecture \cite{Marques2012} we have 
\[
4(\beta^3_1-4\pi) >8\pi.
\]
\end{remark}
\begin{proof} 
We note that the equation $ K = \frac 12 H^2 - \frac 12 |A|^2 $ implies $|K|\le \frac12 |A|^2$ and therefore   
 \begin{align*}
\bigg|\int_{ \Sigma} K d\mu \bigg| \leq \int _{ \Sigma} |K| d\mu \leq 4 \pi + \frac \delta  2 .
\end{align*}  
Since $\delta<8\pi$ a result of White \cite{White1987} yields
\begin{align*}
 \int_\Sigma K d\mu  = 0, \pm 4 \pi.
\end{align*}  
As $ \Sigma$ is complete non-compact we must have $ \int_{ \Sigma} K d\mu \in \{0, - 4 \pi\}$. \newline
\textbf{Case 1)} $\int _{ \Sigma} K d\mu = 0 $\newline  
By the Gauss-Bonnet formula for complete non-compact surfaces of finite total curvature (see \eqref{gauss-bonnet}) we conclude
 \begin{align*}
0= \int _{\Sigma} K d\mu = 2 \pi \left ( \chi( \Sigma) - \sum_{ i = 1 } ^ b (k ( a_i) +1) \right)
\end{align*}
where $ k ( a_i )+1$ is the multiplicity of the end $a_i$. Using this formula we see that the only choice is that $\Sigma$ is homeomorphic to $\mbb S^2 \backslash \{a\}$ and $k(a)=0$. 
Hence $f:\Sigma\hookrightarrow \R ^ 3 $ is the immersion of a plane with an end of multiplicity one. 

Since $ \int _{\Sigma } K d\mu = 0$ we also get 
\begin{align*}
\mc {W} (f) = \frac 14 \int_\Sigma |A|^2 d\mu \leq 2 \pi + \frac \delta 4. 
\end{align*} 
As $ \delta < 8\pi$, we see that $ \mc {W} (f) < 4 \pi$, which shows that $ f: \Sigma \hookrightarrow \R ^ 3 $ must be an embedding by \eqref{inversion2}.\newline
\textbf{Case 2)} $\int_{\Sigma} K d\mu = - 4\pi$ \newline  
Using the Gauss- Bonnet formula one more time we now have three choices:
\begin{itemize}
 \item[i)] $\Sigma$ is homeomorphic to a punctured torus $\mbb T^2 \backslash \{a\}$ with $k(a)=0$. In this situation we conclude as above that
\[
 \mc {W} (f) = \frac 14 \int_\Sigma |A|^2 +\frac12 \int_\Sigma K d\mu \le \frac{\delta}{4}< \beta^3_1-4\pi. 
\]
Next we perform an inversion at a point $x \notin f(\Sigma)$ and, using \eqref{inversion2}, we get a conformal immersion $\tilde f \in W^{2,2}$ of a torus satisfying
\[
 \beta^3_1\le \mc {W} (\tilde f) =\mc W (f)+ 4\pi < \beta^3_1,
\]
which yields a contradiction.
\item[ii)] $\Sigma$ is homeomorphic to $\mbb S^2 \backslash \{a_1,a_2\}$ and $k(a_1)=k(a_2)=0$. Using Theorem \ref{1} this implies $\int_\Sigma |A|^2 d\mu \ge 8\pi$ and, by using the same inversion as before, we get a conformal immersion $\tilde f\in W^{2,2}(\tilde \Sigma,\R^3)$ of a sphere with exactly one point of multiplicity two which has two preimage points. Moreover,  the immersion satisfies
\[
 24\pi \le \int_{\tilde \Sigma} |\tilde A|^2 d \tilde{\mu} =\int_\Sigma |A|^2 d\mu +16 \pi \le 24\pi +\delta.
\]
Since $\delta < \delta_{cat}$ we can apply Theorem \ref{rigidity2} to conclude that $\tilde{f}$ must be $W^{2,2}$ close to an inversion of the catenoid, up to a M\"obius transformation and a reparameterisation. 
\item[iii)] $\Sigma$ is homeomorphic to $\mbb S^2 \backslash \{a\}$ with $k(a)=2$. Similar to the previous case we invert the surface and obtain a $W^{2,2}$ branched conformal immersion of the sphere with exactly one branch point of branch order two and which satisfies
 \[
 32\pi \le \int_{\tilde \Sigma} |\tilde A|^2 d \tilde{\mu} =\int_\Sigma |A|^2 d\mu +24 \pi \le 32\pi +\delta.
\]
Theorem \ref{rigidity3} then shows that $\tilde{f}$ has to be $W^{2,2}$ close to an inversion of Enneper`s minimal surface, up to a M\"obius transformation and a reparameterisation.
\end{itemize} 
\end{proof}
Next we present a similar result in higher codimension.
\begin{theorem}
Let $ f : \Sigma \hookrightarrow \R ^ n  $ be a complete, connected, non-compact surface immersed into $ \R ^ n, n \geq 4  $. Assume that  
\begin{align*}
\int_{\Sigma} |A| ^ 2 d\mu \leq 4 \pi + \delta. 
\end{align*}
where $\delta < \min \{ \delta_{Chen}, 4\pi\}$ and $ \delta_{Chen}$ is from Theorem \ref{rigidity4}. Then either $f$ is a conformal immersion of a plane embedded in $ \R^ n $ or, after composition with an inversion $I_x$, $x\notin f(\Sigma)$, $I_x\circ f$ is $ W^{2,2}$ close to an inversion of Chen's minimal graph, modulo M\"obius transformations and reparameterisations.  
\end{theorem}
\begin{proof} 
Using again the estimate $ |K| \leq \frac12 |A| ^ 2 $ we conclude 
\begin{align*}
\left |\int_\Sigma K d\mu \right | \leq \int_\Sigma | K | d\mu \leq \frac 12 \int_\Sigma |A|^ 2 d\mu \leq 2 \pi + \frac \delta 2.  
\end{align*}
Hence as $\delta<4\pi$ and $ f : \Sigma \hookrightarrow \R ^ n , n \geq 4 $ is complete non-compact, a theorem of White \cite{White1987} implies 
\begin{align*}
\int_\Sigma K d\mu =0, -2\pi. 
\end{align*}
\textbf{Case 1)} $\int_\Sigma K d\mu = 0 $.\newline  
By the Gauss-Bonnet formula \eqref{gauss-bonnet} we have 
\begin{align*}
0=\int_\Sigma K d\mu = 2 \pi \left( \chi (\Sigma) - \sum_{ i =1 } ^  b ( k( a _ i ) +1)\right)  .
\end{align*}
Hence we must have that $ \Sigma$ is homeomorphic to $\mbb{S} ^ 2 \backslash\{ a \} $ and $ k ( a ) =0$ that is $ a$ is an end of multiplicity one. 
In particular this shows that $f: \Sigma \hookrightarrow \R ^ n $ is conformal to a plane. Furthermore, we have 
\begin{align*}
\mc W(f) \le \pi + \frac{\delta}{4}<2\pi
\end{align*}
and therefore $ f :\Sigma  \hookrightarrow \R ^ n $ must be an embedding. 
\newline
\textbf{Case 2)} $\int_{\Sigma } K d\mu = - 2 \pi$ \newline     
Note that in this case we must have the lower bound $ \int _{\Sigma} |A| ^ 2 d\mu \geq 4 \pi$ by Theorem \ref{1}.  Furthermore, applying the Gauss-Bonnet formula once more, we conclude
\begin{align*}
-2\pi =\int _\Sigma K d\mu = 2 \pi \left ( \chi(\Sigma ) - \sum_{ i =1 } ^ b (k(a_i) +1 ) \right) 
\end{align*}
and the only possibility is that $\Sigma$ is homeomorphic to $\mbb S^2 \backslash \{a\}$ with $k ( a)  = 1$, that is $a$ is an end of multiplicity two. 
Therefore we may invert the surface at a point $ x \in \R ^ n, x \not \in f ( \Sigma ) $, and we denote the resulting immersion by $ \tilde f = I _ x \circ f : \mbb{S}^ 2   \hookrightarrow\R ^ n $. $\tilde f$ is a $ W^{2,2}$ branched conformal immersion of the sphere with exactly one branch point of branch order one and which satisfies         
\begin{align*}
20\pi \leq  \int_{\tilde \Sigma}  | \tilde A| ^ 2 d \tilde{\mu}\leq 20 \pi + \delta.
\end{align*}
As $\delta < \delta _{ Chen}$ we have that by Theorem \ref{rigidity4} $ \tilde f $ is $W^{2,2}$ close to an inverted Chen's minimal graph, up to a M\"obius transformation and a reparameterisation. 
\end{proof}

\end{document}